\documentclass{IEEEtran}
\usepackage{cite}
\usepackage{amsmath,amssymb,amsfonts}
\usepackage{algorithm}
\usepackage{algorithmic}
\usepackage{graphicx}
\usepackage{textcomp}
\usepackage{bm}
\usepackage{mathrsfs}
\usepackage{graphicx,cite}
\usepackage{amsmath}
\usepackage{amssymb}
\usepackage[dvips]{epsfig}
\usepackage{latexsym}
\usepackage{psfrag}
\usepackage{graphicx,subfigure}
\usepackage{hyperref}
\usepackage{color}
\usepackage{amsfonts,mathrsfs,amsmath,stmaryrd,amssymb,pifont}%
\usepackage{amscd,graphicx,array,texdraw,tikz}%
\usetikzlibrary{arrows,decorations.markings,shapes,shadows,fadings}
\usepackage{mathrsfs}%
\usepackage{array}
\usepackage{enumerate}
\usepackage{times}
\usepackage{ntheorem}

\theoremseparator{.}
\newtheorem{definition}{Definition}
\newtheorem{lemma}{Lemma}
\newtheorem{theorem}{Theorem}
\newtheorem{corollary}{Corollary}
\newtheorem{remark}{Remark}
\newtheorem{assumption}{Assumption}

\allowdisplaybreaks

\begin{document}
\title{Distributed online generalized Nash Equilibrium learning in multi-cluster games: A delay-tolerant algorithm }

\author{Bingqian Liu, Guanghui Wen, \IEEEmembership{Senior Member, IEEE}, Xiao Fang, Tingwen Huang, \IEEEmembership{Fellow, IEEE}, and Guanrong Chen, \IEEEmembership{Life Fellow, IEEE}

    \thanks{This work was supported by the National Key Research and Development
Program of China under Grant No. 2022YFA1004702.}
    \thanks{B. Liu, G. Wen, and X. Fang are with the Laboratory of Security Operation and
Control for Intelligent Autonomous Systems, Department of Systems
Science, School of Mathematics, Southeast University, Nanjing 211189,
P. R. China (e-mails: bqliu@seu.edu.cn; ghwen@seu.edu.cn; fangxiao@seu.edu.cn).}
	\thanks{T. Huang is with the Science Program, Texas A\&M University at Qatar, Doha, Qatar (e-mail: tingwen.huang@qatar.tamu.edu).}
	\thanks{G. Chen is with the Department of Electrical Engineering, City University of Hong Kong, Hong Kong SAR, China (e-mail: eegchen@cityu.edu.hk).}
}

\maketitle
\begin{abstract}
This paper addresses the problem of distributed online generalized Nash equilibrium (GNE) learning for multi-cluster games with delayed feedback information. Specifically, each agent in the game is assumed to be informed a sequence of local cost functions and constraint functions, which are known to the agent with time-varying delays subsequent to decision-making at each round. The objective of each agent within a cluster is to collaboratively  optimize the cluster's cost function, subject to time-varying coupled inequality constraints and local feasible set constraints over time. Additionally, it is assumed that each agent is required to estimate the decisions of all other agents through interactions with its neighbors, rather than directly accessing the decisions of all agents, i.e., each agent needs to make decision under partial-decision information. To solve such a challenging problem, a novel distributed online delay-tolerant GNE learning algorithm is developed based upon the primal-dual algorithm with an aggregation gradient mechanism. The system-wise regret and the constraint violation are formulated to measure the performance of the algorithm, demonstrating sublinear growth with respect to time under certain conditions. Finally, numerical results are presented to verify the effectiveness of the proposed algorithm.
\end{abstract}

\section{Introduction}

In recent years, multi-cluster games have attracted considerable attention from various scientific communities, owing partly to its wide applications in economic markets \cite{Zeng_2019}, transportation networks \cite{Shehory_1998}, healthcare networks \cite{Peng_2009}, among many others. In such games, the agents are grouped  into clusters to facilitate practical applications where each cluster may have its own set of objectives, constraints, and interactions. One challenging issue in addressing multi-cluster games involves designing an algorithm that enables the actions of agents to converge to desired actions, meanwhile mitigating possible  incentive for unilateral deviation among agents \cite{Zhou_2023}. 

Though many interesting results have been established along this line, most focus on the offline scenario, where the local cost functions and constraint functions of agents are time-invariant. However, in practice, the cost function of each agent always evolves over time and the agents do not have prior knowledge of its cost function at current time \cite{Shahrampour_2018,Yi_2020,Yi_2021,Lu_2021,Meng_2021}. In fact, even if not focusing on distributed multi-cluster games, within the broader research field of distributed games, there is significantly less work on online games compared to offline games. Nevertheless, a kind of distributed online GNE learning algorithm is presented in \cite{Lu_2021} for online games subject to time-invariant constraints. Moreover, a class of online mirror descent methods is developed in \cite{Meng_2021} to ensure that dynamic regret
and constraint violation grow sublinearly by choosing appropriate decreasing stepsizes. 

The above-mentioned results have significantly contributed to our understanding of resolving distributed online game problems under various scenarios. However, it is worth mentioning that results in \cite{Lu_2021} and \cite{Meng_2021} are derived under a delay-free condition, i.e., the information of cost functions and constraint functions can be instantaneously obtained by each agent, without any delay. Yet, in practice, feedback information often is revealed to agents with delays subsequent to decision-making at each round in an online game \cite{Hsieh_2021}, \cite{Cao_2021}. Partly motivated by this observation, a distributed online NE learning algorithm is constructed in \cite{Meng_2022}, which take mirror descent and one-point time-invariant delayed bandit feedback approaches. In \cite{Liu_2023}, a distributed learning algorithm is proposed for online aggregative games with delayed feedback information in dynamic environments. The relationships between agents in all the above online decision-making studies are either competitive or cooperative. However, in the context of multi-cluster multi-agent systems, agents within a cluster typically manifest cooperative relationships, whereas those situated in distinct clusters tend to adopt competitive behaviors \cite{Zhou_2022}, \cite{Yu_2023}. Recently, an online delay-free multi-cluster game framework is established in \cite{Yu_2023} for multi-cluster multi-agent systems in the presence of coupled constraints. To the best of our knowledge, there is no report on the distributed online multi-cluster game problem with delayed feedback information.

Motivated by the above discussions, this paper aims to solve the online GNE learning problem for multi-cluster games with  delayed feedback information. Compared to existing online algorithms \cite{Shahrampour_2018,Yi_2020,Yi_2021,Lu_2021,Meng_2021,Zhou_2022,Yu_2023}, where the cost functions and constraints are available immediately to agents after decision-making, this paper incorporates time-varying delays in the availability of the cost functions and constraints to agents following the decision-making process. A noteworthy feature is that feedback delays are time-varying and heterogeneous for agents, which is more realistic but also more challenging. 

Specifically, a general framework that incorporates multi-cluster games in the online setting subject to feedback delays is firstly established. Then, a distributed GNE learning algorithm that combines primal-dual algorithm with aggregation gradient mechanism is designed to deal with the effect of time-varying delays on NE learning using partial-decision information, where the assumption made in some literature that each agent has full-decision information is removed. In the partial-decision information setting, agents are required to estimating the decisions of all competitors through interactions with neighboring agents, rather than having direct access to the decisions of all agents. This framework facilitates distributed decision-making, making it particularly favorable in practical scenarios in the absence of a centralized coordinator \cite{Ye_2017}. 

The main contributions of this paper are summarized as follows. In comparison with the existing algorithms developed for distributed online GNE learning, the algorithm presented in this study demonstrates robustness to time-varying delays in accessing the feedback information of cost functions and constraint functions. To mitigate the effect of time-varying delays, a collection of feedback timestamps is formed for each agent to record the feedback information appearing at each time step and an aggregation gradient mechanism is established for  ascertaining a credible descent direction. Furthermore, the system-wise regret and the constraint violation are formulated  to measure the performance of the proposed algorithm, demonstrating sublinear growth with respect to time under some weak conditions. To the best of our knowledge, this paper represents the first attempt to address the problem of distributed online GNE learning for multi-cluster games in the presence of delayed information feedback. 

The rest of the paper is organized as follows. Section \ref{sec.one} presents the formulation of the online multi-cluster game and imposes some assumptions. In Section \ref{sec.two}, the distributed online delay-tolerant GNE learning algorithm is designed. Section \ref{sec.three} includes the main results on the system-wise regret and the constraint violation. To demonstrate the obtained results, numerical simulations are performed in Section \ref{sec.four}. Finally, section \ref{sec.five} concludes the paper.

\textbf{Notation}.  Let $\mathbb{R}$, $\mathbb{R}^n$ and $\mathbb{R}^{m \times n}$ denote the sets of real numbers, $n$-dimensional real vectors and $m \times n$-dimensional real matrices, respectively. $\mathbb R_{\geq 0}^{m}$ represents the set of $m$-dimensional nonnegative real vectors. $\mathbb{N}_+$ denotes the set of positive integers. For a constant $N>1$, $[N]$ denotes the set $\left\{1,\cdots,N\right\}$. For column vectors $x_{1},\cdots,x_{N}$, $col(x_{k})_{k\in[N]}\triangleq col(x_{1},\cdots,x_{N})=\left[x_{1}^\top,\cdots,x_{N}^\top\right]^\top$. For matrices $A_{1},\cdots,A_{N}$, $diag(A_{k})_{k \in [N]}\triangleq diag(A_{1},\cdots,A_{N})$ represents the diagonal matrix in which the diagonal is composed of $A_{1},\cdots,A_{N}$ sequentially. For a given vector $x \in \mathbb{R}^n$, $\left\|  x \right\|$ represents the standard Euclidean norm of $x$, and $[x]_{k}$ denotes the $k$th element of $x$ and $[x]_{+}\triangleq col(max\left\{0,[x]_{k}\right\})_{k \in [n]}$. 
For a closed convex set $\chi$, ${{\mathcal P}_{\chi}}\left(x\right)$ is the Euclidean projection of the vector $x$ onto $\chi$ and $N_{\chi}(x)=\left\{v|\left<v,y-x\right>\leq 0, \forall y \in \chi \right\}$ is the normal cone to $\chi$ at $x$. For matrices $A$ and $B$, $A \otimes B$ is their Kronecker product. The element on the $i$th row and the $j$th column of $A$ is denoted by ${\left[ A \right]}_{ij}$, $A^{T}$ is the transpose of $A$, $\lambda(A)$ is an eigenvalue and $\lambda_{max}(A)$ is the maximum eigenvalue of matrix $A$. $I_n$ is the $n$-dimensional identity matrix, $\bm{1}_n$ denotes the $n$-dimensional column vector with all entries equal to $1$ and $\bm{0}_n$ denotes the $n$-dimensional column vector with all elements $0$. For sets $S_{1}$ and $S_{2}$, $\left|S_{1}\right|$ denotes the cardinality of $S_{1}$ and $ S_{2}\setminus S_{1}=\left\{x\in S_{2}|x \notin S_{1}\right\}$ represents the set of elements that are in $S_{2}$ but not in $S_{1}$. Given two sequences $\{{{a}_{t}}\}_{t=1}^{\infty }$ and $\{{{b}_{t}}\}_{t=1}^{\infty }$, the notation ${{a}_{t}}=\mathcal{O}({{b}_{t}})$ means that there exists a time instant $t_{0}$ and a positive constant $C$ such that $\left| {{a}_{t}} \right| \le C \left| {{b}_{t}} \right|$ for $t \ge t_{0}$. Given two real numbers $a$ and $b$, $mod(a,b)$ is equal to the remainder of $a$ divided by $b$.
\section{Problem Formulation}\label{sec.one}

Consider an online multi-cluster game 
\begin{equation*}
\Gamma\triangleq\left\{\mathcal{V},f^{t},g^{t},\bm\Omega\right\},
\end{equation*}
where $\mathcal{V}\triangleq[N]$ with $N$ denoting the number of clusters, $f^{t}\triangleq \left\{{f}^{1,t}_{1}, \cdots,{f}^{1,t}_{n_{1}}, \cdots,{f}^{N,t}_{1},\cdots,{f}^{N,t}_{n_{N}}\right\}$ with $f_j^{i,t}$ denoting the local cost function of agent $j$ in cluster $i$ at time $t$, $g^{t}\triangleq \left\{{g}^{1,t}_{1}, \cdots,{g}^{1,t}_{n_{1}}, \cdots,{g}^{N,t}_{1},\cdots,{g}^{N,t}_{n_{N}}\right\}$ with $g_{j}^{i,t}$ denoting the constraint function of agent $j$ in cluster $i$ at time $t$ and $\bm\Omega\triangleq\bm\Omega_{1}\times \cdots \times\bm \Omega_{N}$ with $\bm\Omega_{i}\triangleq \bigg\{\prod \limits_{j=1}^{n_{i}}A_{j}|A_{j}=\Omega_{i}, \forall j \in [n_{i}]\bigg\}$ denoting the local constraint set of cluster $i$.

At each round $t \in \left[T\right]$, the local cost function ${f}^{i,t}_{j}$ and constraint function $g^{i,t}_{j}$ are presented to agent $j$ in cluster $i$ after decision $x_{ij,t}$ is made in its own local constraint set ${\Omega }_{i}$, where $T$ is the time horizon. The aforementioned sequential decision-making process can be viewed as the main feature of online learning, which is totally different from the offline version. In the offline setting, the information associated with local cost functions and constraint functions is time-invariant and available at each iteration. 

In this paper, consider $\Omega_{i} \subseteq \mathbb R$ for simplicity and $n\triangleq\sum_{i=1}^N n_i$ with the number of agents contained in cluster $i$ being denoted as $n_i$. The agents in cluster $i$ collaborately minimize cost function $f^{i,t}:\mathbb R^{n} \to \mathbb R$, which is the sum of their local cost functions $f^{i,t}_{j}$ at time $t$. Thus, the online multi-cluster game can be formulated as the following noncooperative game problem: for all $i\in\mathcal{V}$,
\begin{equation}\label{eq.game}
\begin{aligned}
&\min_{{x}_{i,t}}~f^{i,t}\left(x_{i,t},x_{-i,t}\right)\triangleq\sum_{j=1}^{n_i}f_{j}^{i,t}\left(x_{ij,t},x_{-i,t}\right),
\\
&\textit{s.t.}\quad {x}_{i,t} \in \chi^{i,t}\left(x_{-i,t}\right),
\end{aligned}
\end{equation}
where $x_{i,t}=col\left(x_{i1,t},\cdots ,x_{in_{i},t}\right)\in \mathbb R^{n_{i}}$ is the decision variable of cluster $i$, $x_{t}=col\left(x_{1,t},\cdots,x_{N,t}\right)\in \mathbb R^{n}$ collects the decisions of all the agents, $x_{-i,t}=col\left(x_{1,t},\cdots,x_{i-1,t},x_{i+1,t},\cdots,x_{N,t}\right)$ denotes the decision variables of all clusters except for cluster $i$, and $\chi^{i,t}(x_{-i,t})$ is the feasible set of cluster $i$ at time $t$ denoted as
\begin{equation*}
\begin{aligned}
 \chi^{i,t}\left(x_{-i,t}\right)=\bigg\{&x_{i,t} \in \mathbb R^{n_{i}}\bigg|x_{ij,t}\in \Omega_{i}, \sum_{i=1}^{N}\sum_{j=1}^{n_{i}}g^{i,t}_{j}\left(x_{ij,t}\right)\leq \bm 0_{m},
\\&x_{ij,t}=x_{ik,t},\forall j, k \in \left[n_{i}\right] \bigg\},
\end{aligned}
\end{equation*}
where the vector-valued functions $g^{i,t}_{j}:\mathbb R \to \mathbb R^{m}$ imply that each agent is subject to $m$ coupled constraints. 

In this sense, the decisions of the agents in the same cluster must reach consensus. Furthermore, the strategy of each agent depends on other agents' strategies due to the existence of the coupled constraint functions and the mechanisms of the noncooperative games.
\begin{remark}
\rm{The constraint function considered in this paper is more general than those discussed in \cite{Zeng_2019},\cite{Yue_2022}, where the constraint functions are only related to the leader in each cluster or even not coupled.}
\end{remark}

Here, the goal is to design an online algorithm to learn the time-varying optimal solution of problem (\ref{eq.game}), which is known as GNE. For convenience, let $x^{*}_{t}=col(\bm 1_{n_{1}}\otimes x_{1,t}^{*},\cdots, \bm 1_{n_{N}} \otimes x_{N,t}^{*})$ and $x^{*}_{-i,t}=col(\bm 1_{n_{1}} \otimes x_{1,t}^{*},\cdots,\bm 1_{n_{i-1}} \otimes x_{i-1,t}^{*}, \bm 1_{n_{i+1}} \otimes x_{i+1,t}^{*},\cdots, \bm 1_{n_{N}} \otimes x_{N,t}^{*})$.

\begin{definition}
\rm{\textbf{(GNE)} A vector $x^{*}_{t}\in \chi^{t}$ is said to be a GNE of the online multi-cluster game at time $t$ if and only if the following inequality holds:
\begin{equation*}
f^{i,t}(\bm 1_{n_{i}} \otimes x_{i,t}^{*},x_{-i,t}^*){\leq}f^{i,t}({x}_{i,t},{x}_{-i,t}^*),~\forall x_{i,t}\in \chi^{i,t}(x_{-i,t}^{*}),
\end{equation*}
where $\chi^{t}=\big\{x_{t} \in \mathbb R^{n}\big|x_{ij,t}\in \Omega_{i}, \sum_{i=1}^{N}\sum_{j=1}^{n_{i}}g^{i,t}_{j}\left(x_{ij,t}\right)\leq \bm0_{m},x_{ij,t}=x_{ik,t}, \forall i \in \mathcal V, \forall j, k \in \left[n_{i}\right] \big\}$.}
\end{definition}

To guarantee the existence and uniqueness of GNE at each round, the following common assumptions are imposed on the cost functions and constraint functions.

\begin{assumption}\label{assp.one}
\rm{\textbf{(Convexity and Differentiability)} The cost functions $f_{j}^{i,t}\left(x_{ij,t},x_{-i,t}\right)$ are convex and differentiable with respect to $x_{ij,t}$ for $x_{-i,t}\in \mathbb R^{n-n_{i}}$, $g^{i,t}_{j}\left(x_{ij,t}\right)$ are convex and differentiable for $x_{ij,t}\in \mathbb R$. Moreover, the non-empty sets $\Omega_{i}$ are convex, compact and Slater's constraint qualification is satisfied.} 
\end{assumption}

Some bounded properties can be gained from Assumption \ref{assp.one}. Specifically, there exist constants $K>0$ and $G>0$ such that, for any $x_{ij,t}\in \Omega_{i}$, $i\in\mathcal V$, $j\in \left[n_{i}\right]$,  $t\in\left[T\right]$, 
\begin{align*}
\| f^{i,t}_{j}(x_{ij,t},x_{-i,t})\| \leq K,& \quad \|g^{i,t}_{j}(x_{ij,t})\| \leq K,
\\\| \nabla_{x_{ij,t}}f^{i,t}_{j}(x_{ij,t},x_{-i,t})\| \leq G,& \quad \| \nabla g^{i,t}_{j}(x_{ij,t})\| \leq G.
\end{align*}

Furthermore, there exists a constant $R>0$ such that $\left\|x\right\|\leq R$ holds for all $x \in \Omega_{i}$ due to the compactness of $\Omega_i$.

For problem (\ref{eq.game}), define the time-varying Lagrangian function $\mathcal L^{i,t} :\mathbb R^{n} \times \mathbb R^{n_i} \times \mathbb R^{m} \to \mathbb R$ for each cluster $i$ with multipliers $\lambda_{i,t} \in \mathbb R^{n_i}$, $\mu_{i,t} \in \mathbb R^{m}$, as
\begin{equation}\label{eq.Lagrangian}
\begin{aligned}
\mathcal L^{i,t}\left(x_{t},\lambda_{i,t},\mu_{i,t}\right)=&f^{i,t}\left(x_{i,t},x_{-i,t}\right)+\left(\lambda_{i,t}\right)^\top L_{i}x_{i,t}
\\&+\left(\mu_{i,t}\right)^\top \sum_{i=1}^{N}\sum_{j=1}^{n_{i}}g^{i,t}_{j}\left(x_{ij,t}\right)+\mathbb{I}_{\bm\Omega_{i}}\left(x_{i,t}\right),
\end{aligned}
\end{equation}
where $\mathbb{I}_{\bm\Omega_{i}}$ is the indicator function on the local constraint set $\bm\Omega_{i}$, and $L_{i}$ represents the Laplacian matrix associated with graph $\mathcal G_{i}$ of cluster $i$, which will be explained in detail later.

Under Assumption \ref{assp.one}, $x^{*}_{t}$ is a GNE of the online multi-cluster game at time $t$ if and only if there exist optimal multipliers $\lambda^{*}_{i,t} \in \mathbb R^{n_i}$, $\mu^{*}_{i,t} \in \mathbb R_{\geq 0}^{m}$ such that the following Karush-Kuhn-Tucher (KKT) condition is satisfied for all cluster $i$, namely, 
\begin{equation}\label{eq.KKT}
\begin{cases}
\begin{aligned}
\bm 0_{n_i}\in&\nabla_{x_{i,t}}f^{i,t}\left(\bm 1_{n_{i}}\otimes x_{i,t}^{*},x_{-i,t}^{*}\right)+N_{\bm\Omega_{i}}\left(\bm1_{n_{i}}\otimes x_{i,t}^{*}\right)
\\&+\left(L_{i}\right)^\top\lambda^{*}_{i,t}+\nabla g^{i,t}(x^{*}_{i,t})\mu^{*}_{i,t},
\end{aligned}
\\L_{i}(\bm1_{n_{i}}\otimes x_{i,t}^{*})=\bm 0_{n_{i}},
\\\bm 0_{m}\in N_{\mathbb R_{+}^{m}}\left(\mu^{*}_{i,t}\right)-\sum_{i=1}^{N}\sum_{j=1}^{n_{i}}g^{i,t}_{j}\left(x_{i,t}^{*}\right),
\end{cases}
\end{equation}
where $\nabla g^{i,t}\left(x^{*}_{i,t}\right)\triangleq col((\nabla g^{i,t}_{j}(x^{*}_{i,t}))^{\top})_{j\in[n_{i}]}\in \mathbb R^{n_{i}\times m}$.

The first inclusion and the second equality represent the \textit{stationarity conditions}: 
\begin{align*}
0_{n_i}\in\nabla_{x_{i}}\mathcal L^{i,t}(x^{*}_{t},\lambda^{*}_{i,t},\mu^{*}_{i,t}), \bm 0_{n_i}=\nabla_{\lambda_{i}}\mathcal L^{i,t}(x^{*}_{t},\lambda^{*}_{i,t},\mu^{*}_{i,t}),
\end{align*}
respectively, while the last inclusion implies three conditions: $\mu^{*}_{i,t}\geq \bm0_{m}$, $\sum_{i=1}^{N}\sum_{j=1}^{n_{i}}g^{i,t}_{j}(x_{i,t}^{*})\leq \bm0_{m}$ and the \textit{complementary slackness condition} $(\mu^{*}_{i,t})^{\top}(\sum_{i=1}^{N}\sum_{j=1}^{n_{i}}g^{i,t}_{j}(x_{i,t}^{*}))=0$.

Note that the optimal multipliers $\mu^{*}_{1,t},\cdots,\mu^{*}_{N,t}$ can differ across clusters and the GNEs are not unique. However, some GNEs are less meaningful in practice. If the said GNE is sought, then the result will deviate far from the expected route. Therefore, consider a special type of those, called variational GNE (vGNE). Define the \emph{pseudo-gradient mapping} as
\begin{equation*}
F_{t}(x_{t})=col(\nabla_{x_{i,t}}f^{i,t}(x_{i,t},x_{-i,t}))_{i \in \mathcal V},
\end{equation*}
with which the definition of vGNE is given as follows.

\begin{definition}
\rm{\textbf{(vGNE)} A vector $x^{*}_{t}\in \chi^{t}$ is said to be a vGNE at time $t$ if and only if $x^{*}_{t}$ satisfies the following variational inequality:}
\begin{equation}\label{eq.VI}
F_{t}\left(x^{*}_{t}\right)^{\top}\left(x_{t}-x^{*}_{t}\right)\geq 0,~\forall x_{t} \in \chi^{t} \subseteq \mathbb R^{n}.
\end{equation}
\end{definition}

In terms of the KKT condition of the variational inequality \cite{Auslender_2000}, the Lagrangian multipliers $\mu^{*}_{i,t}$ in the KKT condition related to the vGNE are identical for $i \in \mathcal V$. Since the cost functions $f^{i,t}(x_{i,t},x_{-i,t})$ are convex and differentiable with respect to $x_{i,t}$ according to Assumption \ref{assp.one}, it follows from \cite{Yi_2019} that $x^{*}_{t}$ is the vGNE of the online multi-cluster game problem in (\ref{eq.game}) at time $t$ if and only if there exist $\lambda^{*}_{t} \in \mathbb R^{n}$ and $\mu^{*}_{t}\in \mathbb R^{m}_{\geq 0}$ such that 
\begin{equation}\label{eq.KKT_2}
\begin{cases}
\begin{aligned}
\bm 0_{n}\in F_{t}(x^{*}_{t})+ N_{\bm\Omega}(x^{*}_{t})+\widehat L\lambda^{*}_{t} +G_{t}(x^{*}_{t})\mu^{*}_{t},
\end{aligned}
\\ \widehat Lx^{*}_{t}=\bm0_{n},
\\\bm0_{m}\in N_{\mathbb R_{+}^{m}}(\mu^{*}_{t})-\sum_{i=1}^{N}\sum_{j=1}^{n_{i}}g^{i,t}_{j}(x_{i,t}^{*}),
\end{cases}
\end{equation}
where $N_{\bm\Omega}\left(x^{*}_{t}\right)=N_{\bm\Omega_{1}}(\bm 1_{n_{1}}\otimes x_{1,t}^{*})\times \cdots \times N_{\bm\Omega_{N}}\left(\bm 1_{n_{N}}\otimes x_{N,t}^{*}\right)$, $\widehat L=diag\left(L_{1}, \cdots, L_{N}\right)$, $\lambda^{*}_{t}=col\left(\lambda^{*}_{1,t},\cdots,\lambda^{*}_{N,t}\right)$ and $G_{t}\left(x^{*}_{t}\right)$ is defined as
\begin{align*}
G_{t}\left(x^{*}_{t}\right)\triangleq col\left(\nabla g^{i,t}(x^{*}_{i,t})\right)_{i \in \mathcal V}\in \mathbb R^{n \times m}.
\end{align*}

Under Assumption \ref{assp.two}, the variational inequality (\ref{eq.VI}) has only one solution, which can be proved by contradiction. In other words, there exists a unique vGNE at each round $t\in[T]$, see \cite{Tatarenko_2021} for detailed analysis.

\begin{assumption}\label{assp.two}
\rm{\textbf{(Strongly monotonicity)} The pseudo-gradient mapping $F_{t}(x_{t})$ is $\sigma$-strongly monotone on the set $\bm\Omega$ for $t \in \left[T\right]$, i.e.,
\begin{equation}\label{eq.a2}
(F_{t}(x_{t})-F_{t}(y_{t}))^\top(x_{t}-y_{t})\ge \sigma {\left\|x_{t}-y_{t}\right\|}^{2},~\forall x_{t},y_{t} \in \bm\Omega.
\end{equation}
}
\end{assumption}

To ensure a better performance of the online GNE learning algorithm, assume that the local cost functions $f^{i,t}_{j}(x_{ij,t},x_{-i,t})$ are $l$-smooth with respect to $x_{ij,t}$.

\begin{assumption}\label{assp.three}
\rm{\textbf{($l$-Smoothness)} For $\widetilde{x}_{t}=(x_{i,t},x_{-i,t})$ and $\widetilde{y}_{t}=(y_{i,t},y_{-i,t})$, $\nabla_{x_{ij,t}}f^{i,t}_{j}(x_{ij,t},x_{-i,t})$ are $l$-Lipschitz continuous, i.e.,
\begin{equation*}
\left\| {\nabla_{x_{ij,t}}f^{i,t}_{j}(x_{ij,t},x_{-i,t})}-{\nabla_{x_{ij,t}}f^{i,t}_{j}(y_{ij,t},y_{-i,t})} \right\| \leq l \left\| \widetilde{x}_{t}-\widetilde{y}_{t} \right\|,
\end{equation*}
where $x_{ij,t},y_{ij,t}\in  \Omega_{i}$ and $x_{-i,t},y_{-i,t}\in \mathbb R^{n-n_{i}}$.}
\end{assumption}

Obviously, the accumulated error measured by the differences of local cost functions at the actual decisions $x_{ij,t}$ and the ideal decisions $x_{i,t}^*$ are non-decreasing. For these reasons, the regret of agent $j$ in cluster $i$ to measure the quality of the online algorithm is defined as
\begin{equation*}
\mathcal{R}_{ij}\left(T\right)=\sum_{t=1}^{T}\sum_{k=1}^{n_{i}}f^{i,t}_{k}\left(x_{ij,t},x_{-i,t}^{*}\right)-\sum_{t=1}^{T}\sum_{k=1}^{n_{i}}f^{i,t}_{k}\left(x_{i,t}^{*},x_{-i,t}^{*}\right).
\end{equation*}

Similarly, the system-wise regret is defined as
\begin{equation*}
\mathcal{R}\left(T\right)=\sum_{t=1}^{T}\sum_{i=1}^{N}\sum_{j=1}^{n_{i}}\left(f^{i,t}_{j}\left(x_{ij,t},x_{-i,t}^{*}\right)-f^{i,t}_{j}\left(x_{i,t}^{*},x_{-i,t}^{*}\right)\right).
\end{equation*}

In order to measure the degree that $x_{ij,t}$ violate the constraints, the constraint violation is introduced as
\begin{equation*}
\mathcal CV(T)=\left\|\sum_{t=1}^{T}\sum_{i=1}^{N}\sum_{j=1}^{n_{i}}\left[g^{i,t}_{j}(x_{ij,t})\right]_{+}\right\|.
\end{equation*}

\begin{remark}
\rm{The definition of the constraint violation is tighter than those in \cite{Meng_2022,Liu_2023}, and the compensations from the feasible decisions are not permitted owing to the nonnegativeness of $[g^{i,t}_{j}(x_{ij,t})]_{+}$.}
\end{remark}

An online GNE learning algorithm will be designed below to ensure that the system-wise regret and the constraint violation can grow sublinearly, i.e., 
\begin{equation*}
\lim_{T\to \infty}\frac{\mathcal{R}(T)}{T}=0,\quad \lim_{T\to \infty}\frac{\mathcal{CV}(T)}{T}=0. 
\end{equation*}

It means that, for any $\varepsilon>0$, there exists a time instant $T_{0}>1$, such that $\left|\frac{\mathcal R(T)}{T}\right|<\varepsilon$ for $T >T_{0}$.
Equivalently, there exists a constant $c\in \left[0,1\right)$ such that
\begin{equation*}
\mathcal{R}(T)=\mathcal{O}(T^{c}),\quad \mathcal{CV}(T)=\mathcal{O}(T^{c}).
\end{equation*}

Furthermore, each agent interacts with its neighbors over the communication graph to exchange information at each round such that the system-wise regret and the constraint violation stay within a reliable bound. In this paper, assume that all agents in cluster $i$ are connected by an undirected connected graph $\mathcal G_{i}=(\left[n_{i}\right],\mathcal E_{i})$ with $\mathcal E_{i}\subseteq [n_{i}]\times [n_{i}]$ denoting the set of edges. Let each cluster be regarded as a whole, and all these agents deliver information through an undirected connected graph $\mathcal G_{0}=\left(\mathcal V,\mathcal E_{0}\right)$ with $\mathcal E_{0} \subseteq \mathcal V \times \mathcal V$ denoting the set of edges. Thus, if agent $m$ in cluster $i$ can get information from agent $n$ in cluster $j$, then $(i,j)\in \mathcal E_{0}$. Meanwhile, the graph $\mathcal G=\left(\left[n\right],\mathcal E\right)$ represents the communication among all agents, where $\mathcal E \subseteq [n]\times[n]$. The following assumptions associated with graphs $\mathcal G_{i}$ and $\mathcal G_{0}$ are imposed.

\begin{assumption}\label{assp.four}
\rm{\textbf{(Connectivity)} For each $i\in \left\{0,1,\cdots,N\right\}$, the graph $\mathcal G_{i}$ is undirected and connected.}
\end{assumption}

Based on the structure of graph $\mathcal G_{i}$, a mixing matrix $W_{i}$ is constructed for $i \in \mathcal V$. Also, the mixing matrix $ W$ related to the graph $\mathcal G$ is defined similarly.
\begin{assumption}\label{assp.five}\rm{\textbf{(Doubly stochasticity)} The mixing matrices $ W_{i}$ and $W$ satisfy the following conditions.
\begin{enumerate}
\item The mixing matrices $W_{i}$ and $ W$ are doubly stochastic, respectively, i.e.,
\begin{align*}
W_{i}\bm 1_{n_{i}}=\bm1_{n_{i}}, &\quad \quad \quad \bm 1_{n_{i}}^\top W_{i}=\bm 1_{n_{i}}^\top,
\\ W\bm 1_{n}=\bm 1_{n}, &\quad \quad \quad \bm 1_{n}^\top W=\bm 1_{n}^\top.
\end{align*}
\item For each $i \in \mathcal V$, there exists a scalar $a>0$ such that $\left[ W_{i}\right]_{jj}\geq a$ for all $j \in [n_{i}]$, $\left[ W_{i}\right]_{jk}\geq a$ if $(k,j)\in \mathcal E_{i}$ and $\left[  W_{i}\right]_{jk}=0$ otherwise.
\item There exists a scalar $a>0$ such that $\left[  W\right]_{ii}\geq a$ for all $i \in \mathcal [n]$, $\left[  W\right]_{jk}\geq a$ if $(k,j)\in \mathcal E$ and $\left[  W\right]_{jk}=0$ otherwise.
\end{enumerate}}
\end{assumption}


\begin{remark}
\rm{The graph $\mathcal G$ is also undirected and connected under Assumption \ref{assp.four}. Unlike the inter-cluster graph defined in \cite{Yue_2022}, in which communication is limited to leaders only, here the agents from different clusters all have the opportunity to communicate with each other.}
\end{remark}

\begin{remark}
\rm{The online multi-cluster game (\ref{eq.game}) can be viewed as an extension of online optimization, online game and their offline counterparts. Specifically, if there is only one agent in each cluster and the cost function of each agent is time-invariant (or time-varying), it will reduce to (online) noncooperative game problems. In particular, if there exists only one cluster and the cost function of each agent is time-invariant (or time-varying), it reduces to a standard (online) distributed optimization problem where a consensual condition is required. This is distinct from other existing studies \cite{Yu_2023}.}
\end{remark}

\section{Algorithm Development}\label{sec.two}

In this section, a distributed online delay-tolerant GNE learning algorithm is designed in the presence of zero-order oracle delayed feedback information, that is, the cost functions $f^{i,t}_{j}(x_{ij,t},x_{-i,t})$ and constraint functions $g^{i,t}_{j}(x_{ij,t})$ are revealed to agent $j$ in cluster $i$ at time $t+\tau_{ij,t}$ after $x_{ij,t}$ is determined. Here, $\tau_{ij,t}$ is the feedback delay of agent $j$ in cluster $i$ at time $t$. Note that the delays are time-varying, which is of great interest to study due to computational burden, communication latency, and asynchrony.

For $i\in\mathcal V,~j \in [n_{i}],~t\in \left[T\right]$, if $\tau_{ij,t}=1$, then the functions $f^{i,t}_{j}(x_{ij,t},x_{-i,t})$ and $g^{i,t}_{j}(x_{ij,t})$ can be observed at time $t+1$, which reduce to the traditional online settings; if $1<\tau_{ij,t}<T$, the functions $f^{i,t}_{j}(x_{ij,t},x_{-i,t})$ and $g^{i,t}_{j}(x_{ij,t})$ can be obtained at time $t+\tau_{ij,t}$; if $\tau_{ij,t}\geq T$, the active time of the feedback information escapes from the time horizon $T$, thereby none of delayed feedback information can be used by each agent. To avoid the occurrence of the third situation, the following condition is imposed on the delays.

\begin{assumption}\label{assp.six}
\rm{\textbf{(Priori condition)} For $i\in\mathcal V,~j \in [n_{i}],~\exists~t_{0}\in \left[T-1\right]$, such that the active time of feedback information satisfies the following condition:
\begin{equation*}
t_{0} < t_{0}+\tau_{ij,t_{0}} \leq T.
\end{equation*}}
\end{assumption}

\begin{remark}
\rm{Assumption \ref{assp.six} means that there exists at least one time instant $t_{0}\in[T]$ such that $\mathcal S_{ij,t_{0}}$ is nonempty for agent $j$ in cluster $i$. In other words, the agents can receive feedback information during the updates of the primal variables and the dual variables. More importantly, the feedback delay $\tau_{ij,t}$ of each agent $j$ in cluster $i$ at time $t$ considered here can change over time. At present, only constant delays are considered in \cite{Cao_2021,Liu_2023}.}
\end{remark}

The existence of time-varying delays brings significant challenges to the design of the online algorithm. One notable difficulty is to determine the information that becomes available in each round. To address this issue, denote the set of feedback timestamps that is available to agent $j$ in cluster $i$ during period $(t,t+1]$ as $\mathcal S_{ij,t}=\{s|t<s+\tau_{ij,s}\leq t+1, \tau_{ij,s} \in \mathcal T_{ij}\}$ with a local buffer $\mathcal T_{ij}$ storing feedback delay $\tau_{ij,t}$ of agent $j$ in cluster $i$ at each time $t$. That is, the feedbacks presented to agent $j$ in cluster $i$ are the functions $f^{i,k}_{j}$ and $g^{i,k}_{j}$ for $k\in \mathcal S_{ij,t}$ during period $(t,t+1]$. It can be easily verified that $\{\mathcal S_{ij,t}\}_{t \in [T]}$ are matually disjoint, i.e., $\mathcal S_{ij,m}\cap \mathcal S_{ij,n}= \emptyset,~m \neq n,~\forall m,n \in [T]$ and $\mathcal S_{ij,t} \subseteq [t]$. It is worth noting that the non-delayed setting corresponds to the case that $\mathcal S_{ij,t} = \{t\}$ for time $t \in [T]$. 

Notably, in a partial-decision information scenario, the decisions of all the competitors can not be fully accessed by an individual agent. Thus, every agent should hold an auxiliary variable to estimate the strategies of the other agents over graph $\mathcal G$. Let $\bm x_{ij,t}$ represent the estimation of agent $j$ in cluster $i$ to all other agents at time $t$, where $\bm x_{ij,t}=col\left(\bm x^{1}_{ij,t},\cdots,\bm x^{N}_{ij,t}\right)$ with $\bm x^{k}_{ij,t}$ denoting the estimation of agent $j$ in cluster $i$ to the agents in cluster $k$. Furthermore, $\bm x^{p}_{ij,t}=col(\bm x^{p1}_{ij,t}, \cdots, \bm x^{pn_{p}}_{ij,t})$ with $\bm x^{ph}_{ij,t}$ denoting the estimation of agent $j$ in cluster $i$ to  agent $h$ in cluster $p$, where $h \in[n_{p}]$. Note that, $\bm x^{ij}_{ij,t}=x_{ij,t}$.

The designed distributed online delay-tolerant GNE learning algorithm (DODTA) for agent $j$ in cluster $i$ is summarized in Algorithm $1$.

\begin{algorithm}
\caption{Distributed Online Delay-Tolerant GNE Learning Algorithm (DODTA) from the view point of agent $j$ in cluster $i$}
\textbf{Initialize:}
\begin{enumerate}
\item[]Select $x_{ij,1}\in \Omega_{i}$, $\bm x^{ph}_{ij,1}\in \mathbb R$ for any $(p,h)\neq(i,j)$, $\mu_{ij,1}=\bm 0_{m}$ and $z_{ij,1} =\bm 0_{m}$.
\item[]Create the local buffer $\mathcal T_{ij}=\emptyset$.
\end{enumerate}
\textbf{for} $t=1,2,\cdots,T$ \textbf{do}
\begin{enumerate}
\item[]Receive $x_{ik,t}$ from the in-neighbors of agent $j$ in cluster $i$ over graph $\mathcal G_{i}$, $\bm x^{ph}_{qk,t}$, $\mu_{qk,t}$, and $z_{qk,t}$ from the in-neighbors of agent $j$ in cluster $i$ over graph $\mathcal G$.
\item[]Store $\tau_{ij,t}$ into local buffer $\mathcal T_{ij}$.
\item[]Receive the feedback information of $f^{i,s}_{j}$ and $g^{i,s}_{j}$ for $s \in  \mathcal S_{ij,t}$.  
\item[] \textbf{Updates} $\bm x^{ph}_{ij,t+1}$ by (\ref{eq.A1_2}); $x_{ij,t+1}$ by (\ref{eq.A1_3}) and (\ref{eq.A1_4});  $z_{ij,t+1}$ by (\ref{eq.A1_6}); $\mu_{ij,t+1}$ by (\ref{eq.A1_8}) and (\ref{eq.A1_7}).
\item[]Broadcast $x_{ij,t+1}$ to the out-neighbors of agent $j$ in cluster $i$ over graph $\mathcal G_{i}$, $\bm x^{ph}_{ij,t+1}$, $\mu_{ij,t+1}$, and $z_{ij,t+1}$ to the out-neighbors of agent $j$ in cluster $i$ over graph $\mathcal G$.
\end{enumerate}
\textbf{end for}
\end{algorithm}


Let $\left[ W\right]^{qk}_{ij}$ represent the $(\sum_{s=0}^{i-1}n_{s}+j)$th row and the $(\sum_{s=0}^{q-1}n_{s}+k)$th column element of $W$ for $k \in [n_{q}]$, $j \in [n_{i}]$ and $n_{0}\triangleq 0$. The updating procedures of $\bm x^{ph}_{ij,t}$, $x_{ij,t}$, $z_{ij,t}$ and $\mu_{ij,t}$ can be described as
\begin{equation}\label{eq.A1_2}
\bm x^{ph}_{ij,t+1}=\sum_{q=1}^{N}\sum_{k=1}^{n_{q}} \left[W\right]^{qk}_{ij}\bm x^{ph}_{qk,t},~(p,h)\neq(i,j),
\end{equation}
\begin{equation}\label{eq.A1_3}
x_{ij,t+1}=\mathcal P_{\Omega_{i}}\bigg(\sum_{k=1}^{n_{i}}\left[  W_{i}\right]_{jk}\bigg(x_{ik,t}-\alpha_{t}\sum_{s \in \mathcal S_{ik,t}} y^{s}_{ik,t}\bigg)\bigg),
\end{equation}
\begin{equation}\label{eq.A1_4}
y^{s}_{ij,t}=\nabla_{x_{ij,t}} f^{i,s}_{j}(x_{ij,t},\bm x^{-i}_{ij,t})+(\mu_{ij,t})^\top\nabla \left[g^{i,s}_{j}(x_{ij,t})\right]_{+},
\end{equation}
\begin{equation}\label{eq.A1_6}
z_{ij,t+1}=z_{ij,t}-\beta_{t}\sum_{q=1}^{N}\sum_{k=1}^{n_{q}} \left[W\right]^{qk}_{ij}\left(\mu_{ij,t}-\mu_{qk,t}\right),
\end{equation}
\begin{equation}\label{eq.A1_8}
\widetilde{z}_{ij,t}=\sum_{q=1}^{N}\sum_{k=1}^{n_{q}} \left[ W\right]^{qk}_{ij}\left(z_{ij,t}-z_{qk,t}\right),
\end{equation}
\begin{equation}\label{eq.A1_7}
\mu_{ij,t+1}=\bigg[\mu_{ij,t}+\gamma_{t}\bigg(\widetilde{z}_{ij,t}-\sigma_{t}\sum_{s \in \mathcal S_{ij,t}} \left[g^{i,s}_{j}\left(x_{ij,t}\right)\right]_{+}\bigg)\bigg]_{+},
\end{equation}
where $\alpha_{t}$, $\beta_{t}$ and $\gamma_{t}$ represent the step size sequences,  $\sigma_{t}=4^{t-T}$, $\bm x^{-i}_{ij,t}=col\left(\bm x^{1}_{ij,t},\cdots,\bm x^{i-1}_{ij,t},\bm x^{i+1}_{ij,t},\cdots,\bm x^{N}_{ij,t}\right)$, $\mu_{ij,t}$ is a dual multiplier related to the inequality constraints to track $\mu^{*}_{t}$ and $z_{ij,t}$ is a local auxiliary variable to ensure the consensus of the local multipliers $\mu_{ij,t}$ for all agents and estimate the contributions of the other agents under the coupled inequality constraints. With a little abuse of terminology and notation, $\nabla [g^{i,s}_{j}(x_{ij,t})]_{+}=col\big(\big[\nabla [g^{i,s}_{j}(x_{ij,t})]_{+}\big]_{k}\big)_{k\in[m]}$ is defined as
\begin{equation}
\big[\nabla [g^{i,s}_{j}(x_{ij,t})]_{+}\big]_{k}=\left\{\begin{aligned}&\nabla\big[g^{i,s}_{j}(x_{ij,t})\big]_{k}, &\big[g^{i,s}_{j}(x_{ij,t})\big]_{k} \geq 0, \\&0, &\big[g^{i,s}_{j}(x_{ij,t})\big]_{k} < 0.\end{aligned}\right.
\end{equation}

If there exists $t_{0}\in[T]$ such that $\mathcal S_{ij,t_{0}}= \emptyset$, then the agents update $x_{ij,t_{0}+1}$ and $\mu_{ij,t_{0}+1}$ without any feedback information. Note that the terms $\sum_{s \in \mathcal S_{ik,t_{0}}} y^{s}_{ik,t_{0}}$ in (\ref{eq.A1_3}) and $\sum_{s \in \mathcal S_{ij,t_{0}}} [g^{i,s}_{j}\left(x_{ij,t_{0}}\right)]_{+}$ in (\ref{eq.A1_7}) are equal to $0$ and $\bm 0_{m}$ respectively.

\begin{remark}
\rm{The update of each agent relies on its local variables, cost functions and constraint functions in the past time. Simultaneously, every agent needs to communicate with its neighbors in the same cluster or even other clusters. Furthermore, none of the information will be shared among agents directly, thus Algorithm 1 designed here is truly  distributed. Compared with \cite{Yue_2022}, there is no need to set a virtual center to gather information in each cluster, consequently the complexity of the algorithm is reduced. }
\end{remark}

For simplicity, denote $col(\cdot)_{i \in \mathcal V, j \in [n_{i}]}$ and $diag(\cdot)_{i \in \mathcal V, j \in [n_{i}]}$ as $col(\cdot)$ and $diag(\cdot)$, respectively. 

Let $\bm x_{t}=col(\bm x_{ij,t})$, $x_{t}=col(x_{ij,t})$, $z_{t}=col(z_{ij,t})$, $\mu_{t}=col(\mu_{ij,t})$, $\widetilde F_{t}=col(\sum_{s \in \mathcal S_{ij,t}}\nabla_{x_{ij,t}}f^{i,s}_{j}(x_{ij,t},\bm x^{-i}_{ij,t}))$, ${G}_{t}=col(\sum_{s \in \mathcal S_{ij,t}}[g^{i,s}_{j}(x_{ij,t})]_{+})$ and $\widetilde G_{t}=col(\sum_{s \in \mathcal S_{ij,t}}\nabla [g^{i,s}_{j}(x_{ij,t})]_{+})$. Then, (\ref{eq.A1_2})-(\ref{eq.A1_7}) can be compactly written as
\begin{equation}\label{CAl_1}
\bm x_{t+1}=\overline W \bm x_{t},
\end{equation}
\begin{align}\label{CAl_2}
x_{t+1}=\mathcal P_{\bm\Omega}\left(\widehat W\left(x_{t}-\alpha_{t}\left(\widetilde F_{t}+\widehat \mu_{t} \widetilde G_{t}\right)\right)\right),
\end{align}
\begin{equation}\label{CAl_4}
z_{t+1}=z_{t}-\beta_{t}\overline L\mu_{t},
\end{equation}
\begin{equation}\label{CAl_5}
\mu_{t+1}=\left[\mu_{t}+\gamma_{t}\left(\overline Lz_{t}-\sigma_{t}{G}_{t}\right)\right]_{+},
\end{equation}
where $\overline W= W \otimes I_{n}$, $\widehat W=diag\left(W_{i}\right)_{i \in \mathcal V}$, $\widehat \mu_{t}=diag\left(\mu_{ij,t}^\top\right)$ and $\overline L=L \otimes I_{m}$ with $L\in\mathbb R^{n\times n}$ denoting the Laplacian matrix of graph $\mathcal G$, i.e., $[L]_{ij}^{qk}=-[W]_{ij}^{qk}$ if $(i,j)\neq(q,k)$ and $[L]_{ij}^{ij}=\sum_{q=1}^{N}\sum_{k=1}^{n_{q}}[W]_{ij}^{qk}$ otherwise. 
\section{Main Results}\label{sec.three}
In this section, the upper bounds of the system-wise regret and the constraint violation are established. To do so, some necessary Lemmas are first presented.

\begin{lemma}\label{lemma_2}
\rm{Let Assumptions \ref{assp.one}, \ref{assp.four}-\ref{assp.six} hold. For any $i,p \in \mathcal V$, $j \in [n_{i}]$ and $q \in [n_{p}]$, it holds that 
\begin{align*}
&\sum_{t=1}^{T}\left\|\bm x^{ij}_{pq,t}-x_{ij,t}\right\|
\\&\leq\frac{4\sqrt{n}R}{1-\lambda^{-}}+\frac{\sqrt{n}}{1-\lambda^{-}}\sum_{k=1}^{n_{i}}\sum_{t=1}^{T}\left\|x_{ik,t}-x_{ij,t}\right\|
\\&\quad+\frac{\sqrt{n}G}{1-\lambda^{-}}\sum_{k=1}^{n_{i}}\sum_{t=1}^{T}\sum_{s\in \mathcal S_{ij,t}}\alpha_{t}(1+\left\|\mu_{ik,t}\right\|),
\end{align*}
where $\lambda^{-}=\underset{i\in\mathcal V,j\in [n_{i}]}{max}\left\{\lvert\lambda_{max}(W_{ij}^{-})\rvert\right\}<1$ with $W^{-}_{ij}$ representing the submatrix of $W$ in which the element in the $(\sum_{s=1}^{i-1}n_s+j)$th row and the $(\sum_{s=1}^{i-1}n_s+j)$th column is removed. }
\end{lemma}
\begin{proof}
See Appendix A.
\end{proof}

To obtain the bounds of the dual variables $\mu_{ij,t}$ and the auxiliary variables $z_{ij,t}$, the following Assumption is needed.
\begin{assumption}\label{assp.seven}
\rm{\textbf{(Step size criterions)} The sequences $\left\{\alpha_{t}\right\}_{t \in \left[T\right]}$, $\left\{\beta_{t}\right\}_{t \in \left[T\right]}$ and $\left\{\gamma_{t}\right\}_{t \in \left[T\right]}$ satisfy the following conditions:
\begin{enumerate}[(a)]
\item{$0<\alpha_{t}\leq\alpha_{t-1}\leq 1$, $0<\beta_{t}\leq\beta_{t-1}\leq 1$,~$0<\gamma_{t}\leq\gamma_{t-1}\leq 1$}.
\item{$0<\beta_{t}\gamma_{t-1}\leq\beta_{t-1}\gamma_{t}\leq1$}.
\end{enumerate}}
\end{assumption}

In the next lemma, common bounds of the dual variables $\mu_{ij,t}$ and the auxiliary variables $z_{ij,t}$ are established by mathematical induction.
\begin{lemma}\label{lemm_1}
\rm{Let Assumptions \ref{assp.one}, \ref{assp.four}-\ref{assp.seven} hold. For any $i \in \mathcal V$, $j \in [n_{i}]$ and $t \in [T]$, it holds that 
\begin{equation}\label{eq.lemm_21}
\begin{aligned}
\left\|\mu_{ij,t}\right\|\leq\frac{\sigma_{t}(t-1)K}{\Delta_{t-1}},
\end{aligned}
\end{equation}
\begin{equation}\label{eq.lemm_22}
\begin{aligned}
\left\|z_{ij,t}\right\|\leq\frac{\sigma_{t}(t-1)K}{\Delta_{t-1}},
\end{aligned}
\end{equation}
where $\Delta_{t}=\beta_{t}\gamma_{t}^{-1}$.
Further, let $\lambda=\lambda_{max}\left( W-\frac{1}{n}\bm1_{n}\bm1^{T}_{n}\right)$ and $\bar{\mu}_{t}=\frac{1}{n}\sum_{i=1}^{N}\sum_{j=1}^{n_{i}}\mu_{ij,t}$. Then, for $t \geq 2$, 
\begin{equation}\label{eq.lemm_23}
\begin{aligned}
\left\|\mu_{ij,t}-\bar{\mu}_{t}\right\|\leq\sum_{s=0}^{t-2}\frac{5\sqrt{n}\lambda^{s}\sigma_{t-s-1}\gamma_{t-s-2}\left(t-s-1\right)K}{\beta_{t-s-2}}.
\end{aligned}
\end{equation}}
\end{lemma}
\begin{proof}
See Appendix B.
\end{proof}

\begin{lemma}\label{lemm_3}
\rm{Let Assumptions \ref{assp.one}, \ref{assp.four}-\ref{assp.seven} hold. For any $i \in \mathcal V$, $j \in \left[n_{i}\right]$, $t \in \left[T\right]$ and $\mu \in \mathbb R^{m}_{\geq 0}$, it holds that 
\begin{align*}
&-\sum_{t=1}^{T}\sum_{i=1}^{N}\sum_{j=1}^{n_{i}}\sum_{s \in \mathcal S_{ij,t}}\left(\left[g^{i,s}_{j}\left(x_{ij,t}\right)\right]_{+}\right)^\top\mu-\frac{n}{2}\left\|\mu\right\|^{2}
\\&\leq 4nK^{2}\sum_{t=2}^{T}\frac{\sigma_{t}^{2}\gamma_{t-1}^{4}\left(t-1\right)^{2}}{\beta_{t-1}^{2}}+K^{2}\sum_{t=1}^{T}\sum_{i=1}^{N}\sum_{j=1}^{n_{i}}\sigma_{t}^{2}\gamma_{t}^{2}\left|\mathcal S_{ij,t}\right|^{2}
\\&\quad+10K^{2}n\sqrt{n}\sum_{t=2}^{T}\frac{\sigma_{t}\gamma_{t-1}^{2}\left(t-1\right)}{\beta_{t-1}}\sum_{s=0}^{t-2}\frac{\lambda^{s}\sigma_{t-s-1}\left(t-s-1\right)}{\Delta_{t-s-2}}.
\end{align*}}
\end{lemma}
\begin{proof}
See Appendix C.
\end{proof}

For each agent $j$ in cluster $i$, let $\mathcal U_{ij,t}$ represent the set of time instants contained in $[t]$ at which the feedback information can not be gained during period $[1,t+1]$. In other words, $\mathcal U_{ij,t}=[t]\setminus \cup_{h=1}^{t}\mathcal S_{ij,h}$. The maximum of the number of the unavailable time instants during each time period $[1,t]$ is denoted as $\left|\mathcal U_{ij,[T]}\right|$, i.e.,
\begin{align*}
\left|\mathcal U_{ij,[T]}\right|=\mathop{\max}\limits_{t\in [T]}\left|\mathcal U_{ij,t}\right|.
\end{align*}

Furthermore, $\mathcal T_{ij,T}=\sum_{t=1}^{T}\sum_{s\in \mathcal S_{ij,t}}\tau_{ij,s}$ represents the sum of delays of the feedback information that occur during period $[1,T+1]$ and $\Phi_{T}=\sum_{t=1}^{T}\sum_{i=1}^{N}\left\|x^{*}_{i,t+1}-x^{*}_{i,t}\right\|$ reflects the extent of the variation for vGNE at each time. To proceed,  the following assumptions are imposed.

\begin{assumption}\label{assp.eight}
\rm{\textbf{(Common upper boundedness)}
For all agent $j$ in cluster $i$, there exists a common constant $c>0$ such that $\left|\mathcal S_{ij,t}\right|\leq c$ holds for all $t \in [T]$.}
\end{assumption}

\begin{assumption}\label{assp.nine}
\rm{\textbf{(Sublinear condition)}
For agent $j$ in cluster $i$, both $\left|\mathcal U_{ij,[T]}\right|$ and the path violation $\Phi_{T}$ are sublinear, i.e. there exist some constants $0\leq u<1$ and $0\leq \phi<1$ such that $\left|\mathcal U_{ij,[T]}\right|\leq \mathcal O(T^{u})$ and $\Phi_{T}\leq \mathcal O(T^{\phi})$.}
\end{assumption}

\begin{assumption}\label{assp.ten}
\rm{\textbf{(Delay condition)} 
For agent $j$ in cluster $i$, there exists a constant $1\leq \tau<2-\phi$ such that $\mathcal T_{ij,T}\leq \mathcal O(T^{\tau})$.}
\end{assumption}

\begin{remark}
\rm{Assumptions \ref{assp.eight}-\ref{assp.ten} mean that the feedback information at every moment getting closer to be known to the agent in limited time as $T$ increases. Besides, the vGNE at each moment cannot deviate too much from the one at the previous moment. Particularly, traditional online games without delays correspond to the case with $c=1,~\left|\mathcal U_{ij,[T]}\right|=0,~\mathcal T_{ij,T}=T$ and the online games with constant delays $2<t_{0}<T$ correspond to the case with $c=1,~\left|\mathcal U_{ij,[T]}\right|=t_{0}-1,~\mathcal T_{ij,T}=t_{0}(T-t_{0}+1)$.}
\end{remark}

Next, the convergence of Algorithm 1 is analyzed, for which some lemmas are first prepared.
\begin{lemma}\label{lemma_6}
\rm{Suppose Assumptions \ref{assp.one}, \ref{assp.four}-\ref{assp.eight} hold. Let $\alpha_{t}=T^{-a_{1}}$, $\beta_{t}=T^{-a_{2}}$ and $\gamma_{t}=T^{-a_{3}}$ with $0<a_{1}<1$ and $0<a_{2}<1<a_{3}<2$ for all $t \geq 0$ in Algorithm 1. Then, it holds that 
\begin{align*}
&\sum_{t=1}^{T}\left\|x_{ij,t}-\bar{x}_{i,t}\right\|\leq b_{1}+b_{2}T^{1-a_{1}}+b_{3}T^{2+a_{2}-a_{1}-a_{3}},
\end{align*}
where $\bar x_{i,t}=\frac{1}{n_{i}}\sum_{j=1}^{n_{i}}x_{ij,t}$, $\psi=1-\frac{a}{4n_{i}^{2}}$ with $a$ being given in Assumption \ref{assp.five}, $b_{1}=\frac{nR}{\psi^{3}(1-\psi)}+\left\|x_{ij,1}-\bar{x}_{i,1}\right\|$, $b_{2}=\frac{3ncG}{\psi^{3}(1-\psi)}$ and $b_{3}=b_{2}K$.}
\end{lemma}
\begin{proof}
See Appendix D.
\end{proof}

\begin{lemma}\label{lemma_7}
\rm{Suppose Assumptions \ref{assp.one}-\ref{assp.ten} hold. Let $\alpha_{t}=T^{-a_{1}}$, $\beta_{t}=T^{-a_{2}}$ and $\gamma_{t}=T^{-a_{3}}$ with $0<a_{1}<1$ and $0<a_{2}<1<a_{3}<2$ for all $t \geq 0$ in Algorithm 1. Then, it holds that 
\begin{align*}
&\sum_{t=1}^{T}\left\|\bar x_{i,t}-x^{*}_{i,t}\right\|
\\&\leq \mathcal{O}\left(T^{\frac{1+u+\phi}{2}}+T^{\frac{1+a_{1}+\phi}{2}}+T^{\frac{1+\tau-a_{1}}{2}}+T^{\frac{3+a_{2}-a_{3}}{2}}\right)
\\&\quad+\mathcal{O}\left(T^{2+\frac{2a_{2}-a_{1}-2a_{3}}{2}}+T^{\frac{2+\tau+a_{2}-a_{1}-a_{3}}{2}}\right)+\sqrt{\frac{\vartheta \mathcal CV(T)T}{\sigma}}.
\end{align*}}
\end{lemma}
\begin{proof}
See Appendix E.
\end{proof}
\begin{theorem}\label{thm1}
\rm{Suppose Assumptions \ref{assp.one}-\ref{assp.ten} hold. Let $\alpha_{t}=T^{-a_{1}}$, $\beta_{t}=T^{-a_{2}}$ and $\gamma_{t}=T^{-a_{3}}$ with $0<a_{1}<1-\phi$ and $0<a_{2}<1<a_{3}<2$ for all $t \geq 0$ in Algorithm 1. Then, it holds that 
\begin{align*}
&\mathcal R\left(T\right)= \mathcal{O}\left(T^{max\left\{\frac{1+u+\phi}{2},\frac{1+a_{1}+\phi}{2},\frac{1+\tau-a_{1}}{2}\right\}}\right)
\\&\quad\quad\quad +\mathcal{O}\left(T^{max\left\{\frac{3+a_{2}-a_{3}}{2},\frac{4+2a_{2}-a_{1}-2a_{3}}{2},\frac{2+\tau+a_{2}-a_{1}-a_{3}}{2}\right\}}\right),
\\&\mathcal{CV}\left(T\right)= \mathcal{O}\left(T^{max\left\{1+a_{2}-\frac{3}{2}a_{3},\frac{1}{2}-a_{3}\right\}}\right).
\end{align*}}
\end{theorem}
\begin{proof}
See Appendix F.
\end{proof}

\begin{remark}
\rm{Theorem \ref{thm1} implies that the system-wise regret $\mathcal R(T)$ and the constraint violation $\mathcal CV(T)$ can grow sublinearly under the condition that the parameters $a_{1}$, $a_{2}$ and $a_{3}$ satisfy the following constraints:
\begin{align*}
\begin{cases}
u+\phi<1,
\\\phi+a_{1}<1,
\\\tau-a_{1}<1,
\\1+a_{2}-a_{3}<0,
\\2+2a_{2}-a_{1}-2a_{3}<0,
\\\tau+a_{2}-a_{1}-a_{3}<0.
\end{cases}
\end{align*} 

Clearly, if $a_{2}$ is small enough to approach $0$ and $a_{3}$ is sufficiently large, the above rules can easily be satisfied. Meanwhile, if $a_{1}$ becomes large, then $\phi+a_{1}$ will increase and $\tau-a_{1}$ will decrease. Therefore, determining an appropriate value for $a_{1}$ that balances the relationship between the above terms involves a tradeoff. Moreover, if the vGNE at each time fluctuates violently, then $\phi$ will be close to $1$. In this case, $\tau$ must approach $1$, because if not then $a_{1}$ will not exist.}
\end{remark}

\begin{corollary}\label{col_1}
\rm{Suppose Assumptions \ref{assp.one}-\ref{assp.ten} hold. Let 
\begin{align*}
\alpha_{t}=\frac{1}{T^{\frac{3}{4}}},~\beta_{t}=\frac{1}{T^{\frac{1}{4}}},~\gamma_{t}=\frac{1}{T^{\frac{3}{2}}},
\end{align*}
for all $t \geq 0$ in Algorithm 1, and set the parameters mentioned in Assumptions \ref{assp.nine} and \ref{assp.ten} as $u=\frac{1}{4}$, $\phi=\frac{1}{16}$ and $\tau=1$. Then, it holds that  
\begin{align*}
\mathcal R\left(T\right)=\mathcal{O}\left(T^{\frac{29}{32}}\right),~\mathcal{CV}\left(T\right)=\mathcal{O}\left(T^{-1}\right).
\end{align*}}
\end{corollary}
\begin{proof}
Substituting the corresponding values of the parameters into Theorem \ref{thm1}, the proof is complete.
\end{proof}

\begin{remark}
\rm{It can be seen from Theorem \ref{thm1} that the parameters $\tau$ and $u$ have effect on the convergence of  the system-wise $\frac{\mathcal R(T)}{T}$, but not $\frac{\mathcal CV(T)}{T}$. The upper bound $c$ affects both convergence rates of $\frac{\mathcal CV(T)}{T}$ and $\frac{\mathcal R(T)}{T}$. Moreover, time-varying delays affect both of the convergence of $\frac{\mathcal CV(T)}{T}$ and $\frac{\mathcal R(T)}{T}$, which further confirms our conclusion.} 
\end{remark}

\section{Numerical simulations}\label{sec.four}

In this section, a numerical example is simulated to demonstrate the effectiveness of the proposed distributed online delay-tolerant GNE learning algorithm. The performances of the algorithm with different types of delays are tested. All agents communicate over graph $\mathcal G$ in Fig. \ref{fig_1} (Note that each agent has self-loop, which is hidden for simplicity). Specifically, agent $1$, agent $2$ and agent $3$ are in the first cluster, agent $4$, agent $5$ and agent $6$ are in the second cluster and agent $7$ alone in the third cluster. Furthermore, cluster $1$, cluster $2$ and cluster $3$ communicate through $\mathcal G_{0}$, which  is composed of agent $3$, agent $4$ and agent $7$. Obviously, $N=3,~n=7,~n_{1}=3,~n_{2}=3,~n_{3}=1$.

\begin{figure}[H]
\centering
\centerline{\includegraphics[scale=0.7]{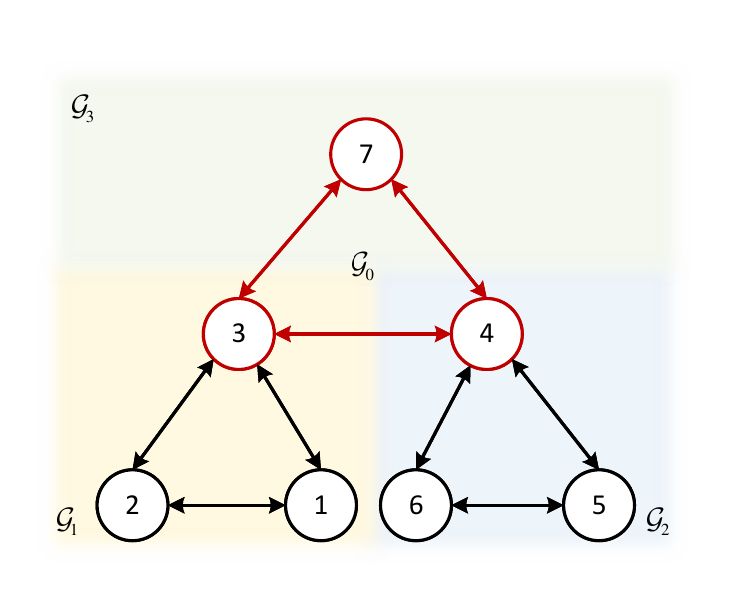}}
\caption{The communication network of the online multi-cluster game.}\label{fig_1}
\end{figure}

Let $h_{1}(t)=\left|sin(0.005t)\right|$ and the cost functions of the agents in each cluster be given as
\begin{align*} 
&f^{1,t}_{1}(x_{t})=20h_{1}(t)(x_{11,t}-5t^{-1})^2+x_{31,t}, 
\\&f^{1,t}_{2}(x_{t})=20h_{1}(t)(x_{12,t}-5t^{-1})^2-x_{21,t}+x_{22,t}^2,
\\&f^{1,t}_{3}(x_{t})=15h_{1}(t)(x_{13,t}-5t^{-1})^2+x_{22,t}(x_{22,t}-x_{21,t}),
\\&f^{2,t}_{1}(x_{t})=20h_{1}(t)(x_{21,t}-5t^{-1})^2+x_{11,t},
\\&f^{2,t}_{2}(x_{t})=10h_{1}(t)(x_{22,t}-5t^{-1})^2+x_{31,t}(x_{12,t}+x_{13,t}),
\\&f^{2,t}_{3}(x_{t})=20h_{1}(t)(x_{23,t}-5t^{-1})^2,
\\&f^{3,t}_{1}(x_{t})=20h_{1}(t)(x_{31,t}-5t^{-1})^2+(x_{11,t}+x_{23,t}).
\end{align*} 

Moreover, let $h_{2}(t)=sin^{2}(10^{-3}t)$ and the time-varying constraint functions $g^{i,t}_{j}(x_{ij,t})$ be given as 
\begin{align*}
g^{i,t}_{j}(x_{ij,t})=col\left(0.5(x_{ij,t}^2-mh_{2}(t))-10^{2}e^{-(i+j-1)}\right)_{m\in[6]}.
\end{align*}

The parameters of the online algorithm are designed as follows: $T=5000,~\alpha_{t}=t^{-0.98},~\beta_{t}=T^{-0.02},~\gamma_{t}=T^{-1.5}$ and $\Omega_{i}=\left[-10,10\right]$. The initial values of $x_{ij,1}$ and $\bm x^{ph}_{ij,1}$ are set as $10$. The doubly stochastic matrix $W$ and $W_{i}$ are conducted by using Metropolis weights \cite{Koshal_2016}. 

Then, consider two types of delays: constant delays and time-varying delays. In the first case, all agents have the same delay and they cannot receive any information of their own cost functions and constraint functions before time $t_{0}$. When $t>t_{0}$, the feedback information at $t-t_{0}$ is received. For more details on the above case, see \cite{Cao_2021}. Furthermore, six constant values of $t_{0}$: 0, 10, 20, 40, 60, 80 are tested to verify the influences of different constant delays on the convergence property of the online algorithm and the performances are evaluated by $\frac{\mathcal R(T)}{T}$ and $\frac{\mathcal CV(T)}{T}$ in Fig. \ref{fig_2} and Fig. \ref{fig_3}, respectively. It can be seen that the longer the delays are, the slower the $\frac{\mathcal R(T)}{T}$ converges, and the same is true for $\frac{\mathcal CV(T)}{T}$. Fig. \ref{fig_4} shows the trajectory $\frac{\mathcal R_{ij}(T)}{T}$ of each agent $j$ in cluster $i$ with $t_{0}=10$, which demonstrates that the regret of each agent grows sublinearly since $\frac{\mathcal R_{ij}(T)}{T}$ converges to $0$.

\begin{figure}
\centering
\centerline{\includegraphics[width=0.55\textwidth]{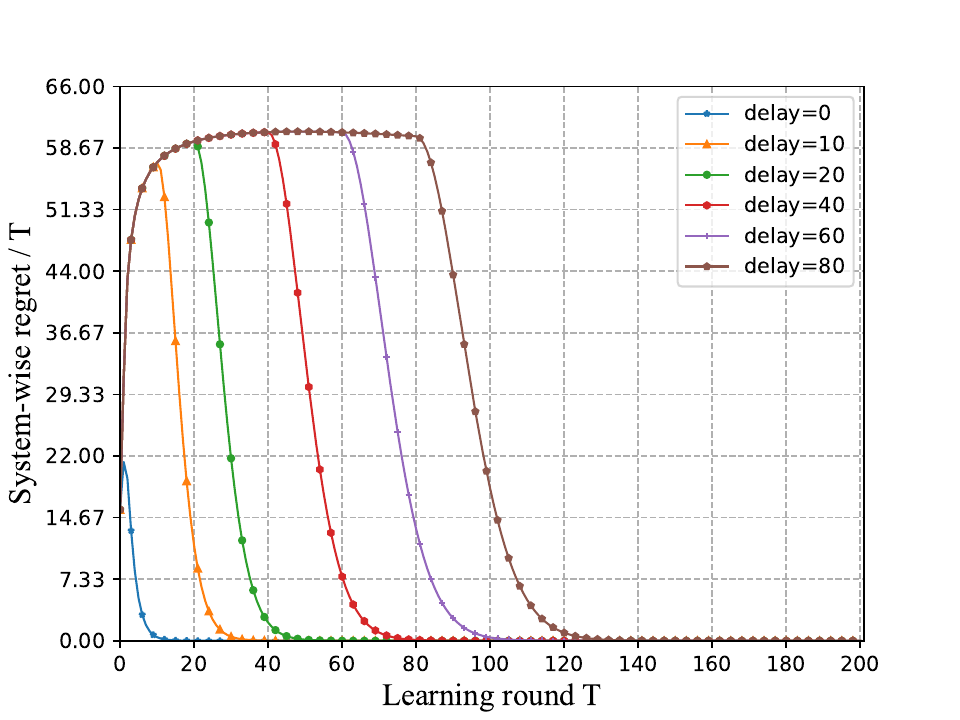}}
\caption{The trajectories of $\mathcal R(T)/T$ with different constant delays.}\label{fig_2}
\end{figure}

\begin{figure}
\centering
\centerline{\includegraphics[width=0.55\textwidth]{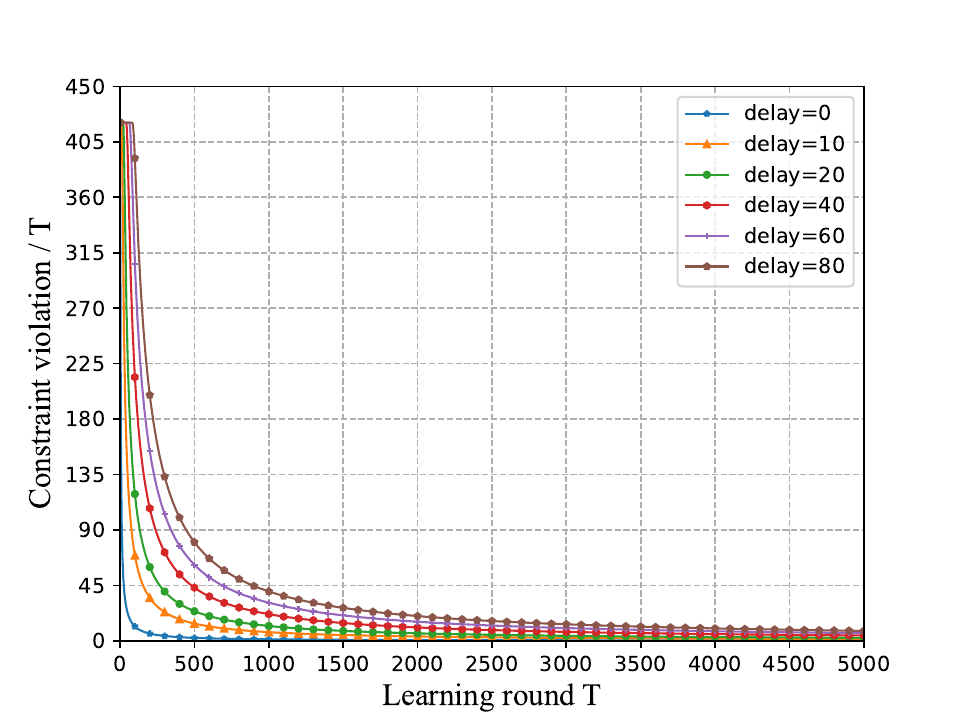}}
\caption{The trajectories of $\mathcal CV(T)/T$ with different constant delays.}\label{fig_3}
\end{figure}

\begin{figure}
\centering
\centerline{\includegraphics[width=0.55\textwidth]{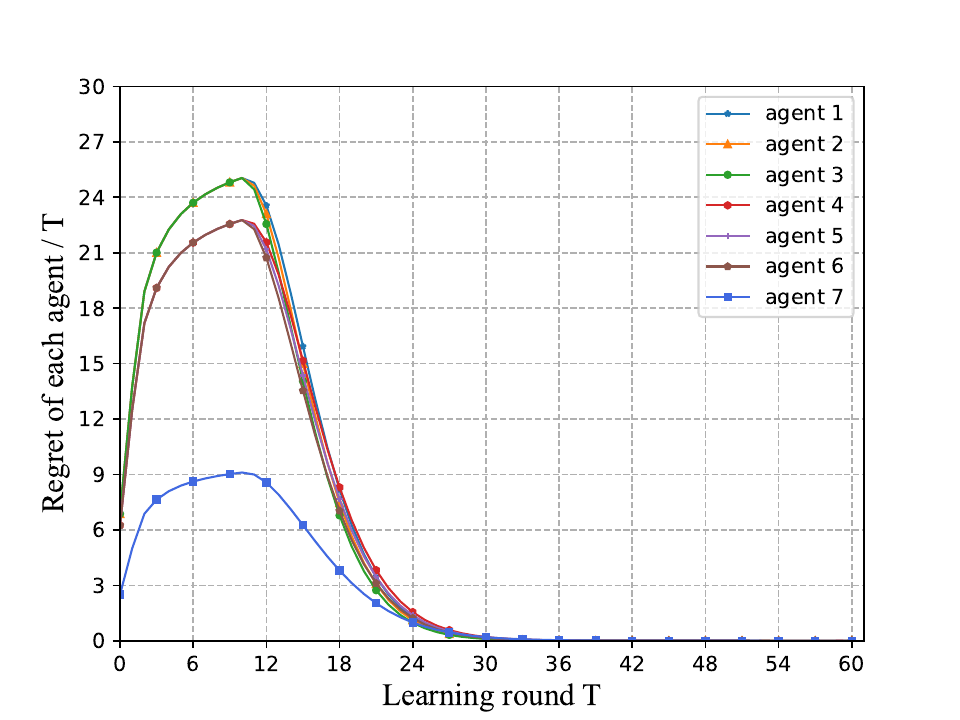}}
\caption{The trajectories of $\mathcal R_{ij}(T)/T$ for each agent with delay $t_{0}=10$.}\label{fig_4}
\end{figure}

In the second case, the delays can change over time and three types of delays are tested, respectively. 

Type 1: This provides feedback information for a historical time period on regular intervals. Formally, let $S_{1}(t_{1})=\left\{t\in[T]|mod(t,t_{1})=1, t\neq 1\right\}$ with $t_{1}\neq 0$, and the set of feedback timestamps $\mathcal S_{ij,t}$ be defined as
\begin{equation}\label{e1}
\mathcal S_{ij,t}=\left\{
\begin{aligned}
\emptyset,\quad &t \notin S_{1}(t_{1}),\\
\cup_{i=1}^{t_{1}}\left\{t-i\right\},\quad &t \in S_{1}(t_{1}).\\
\end{aligned}
\right.
\end{equation}

Type 2: This presents feedback information at certain times and the interval period increases linearly with $T$. Let $S_{2}(t_{2})=\big\{t\in[T]|\exists~t_{3} \in [T],s.t.~t=(\sum_{i=1}^{t_{3}}i)t_{2}+1, t\neq 1\big\}$ and $s(t,t_{2})=\big\{t_{2}t_{3}|t=(\sum_{i=1}^{t_{3}}i)t_{2}+1\big\}$, $\forall t \in S_{2}(t_{2})$, where $t_{2}\neq 0$, and the set of feedback timestamps $\mathcal S_{ij,t}$ be defined as
\begin{equation}\label{e2}
\mathcal S_{ij,t}=\left\{
\begin{aligned}
\emptyset,\quad &t \notin S_{2}(t_{2}),\\
\cup_{i=1}^{s(t,t_{2})}\left\{t-i\right\},\quad &t \in S_{2}(t_{2}).\\
\end{aligned}
\right.
\end{equation}

Type 3: The main characteristic of this type is that the delays of all agents are different. Let $t_{4}=2i+j$, and the set of feedback timestamps $\mathcal S_{ij,t}$ be defined as
\begin{equation}\label{e3}
\mathcal S_{ij,t}=\left\{
\begin{aligned}
\emptyset,\quad &1\leq t\leq10t_{4},\\
\left\{t-10t_{4},t\right\},\quad &10t_{4}<t\leq20t_{4},\\
\left\{t\right\},\quad &t>15t_{4}.\\
\end{aligned}
\right.
\end{equation}

It can be verified that $\left|\mathcal S_{ij,t}\right|$, $\left|\mathcal U_{ij,[T]}\right|$ and $\mathcal T_{ij,T}$ for each agent $j$ in cluster $i$ with delays of type 1 and type 3 satisfy Assumptions \ref{assp.eight}-\ref{assp.ten}. But $\left|\mathcal S_{ij,t}\right|$ and $\left|\mathcal U_{ij,[T]}\right|$ with delays of type 2 are linear with $T$, and $\mathcal T_{ij,T}$ is quadratic with $T$, i.e., $\mathcal T_{ij,T}=\mathcal O(T^{2})$, which clearly violates Assumptions \ref{assp.eight}-\ref{assp.ten}. Let $t_{1},t_{2}=30$, and then test the performance of the algorithm for the network with time-varying delays of the three types. 

Fig. \ref{fig_6} and Fig. \ref{fig_7} respectively show the trajectories of $\mathcal R(T)/T$ and $\mathcal CV(T)/T$, from which it can be observed that the trajectories with delays of type 1 and type 3 converge to $0$. However, the trajectory of $\mathcal R(T)/T$ with delays of type 2 is fluctuating near $0$ as can be seen more clearly from Fig. \ref{fig_7} where the trajectory of $\mathcal CV(T)/T$ with delays of type 2 tends to a non-zero constant as $T$ increases. All these mean that the decisions generated by the online algorithm with delays of type 2 oscillate around the time-varying vGNE as $T$ increases. Interestingly, it leads us to ponder whether it is possible to design an effective algorithm to ensure that $\mathcal R(T)/T$ and $\mathcal CV(T)/T$ converge to $0$ when $\tau\geq2$ in Assumption \ref{assp.ten}.

\begin{figure}[H]
\centering
\centerline{\includegraphics[width=0.55\textwidth]{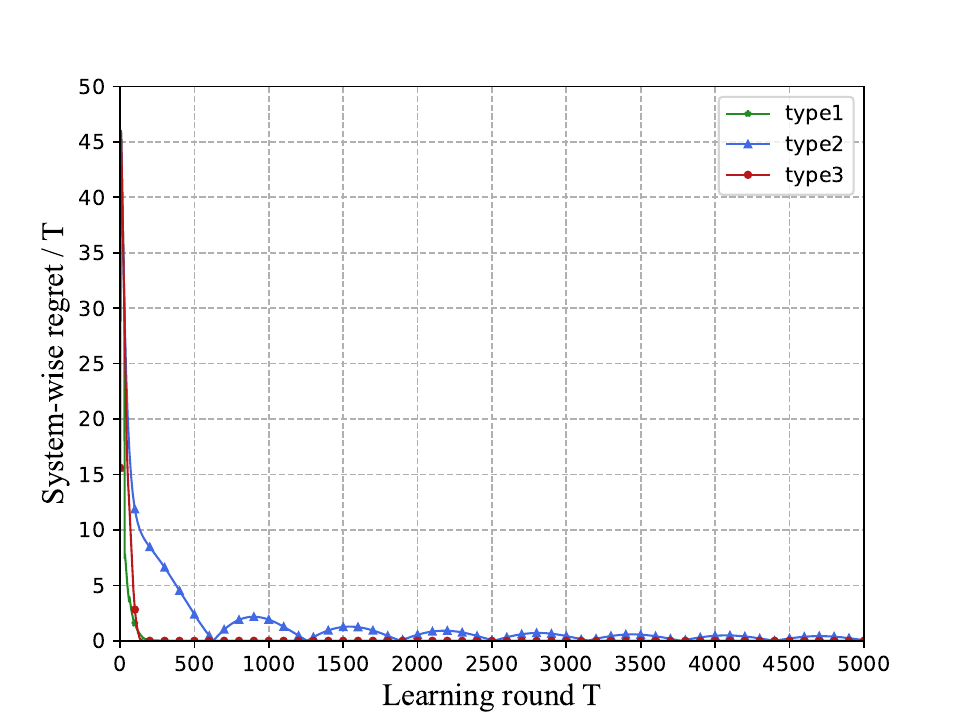}}
\caption{The trajectories of $\mathcal R(T)/T$ with different types of time-varying delays.}\label{fig_6}
\end{figure}

\begin{figure}[H]
\centering
\centerline{\includegraphics[width=0.55\textwidth]{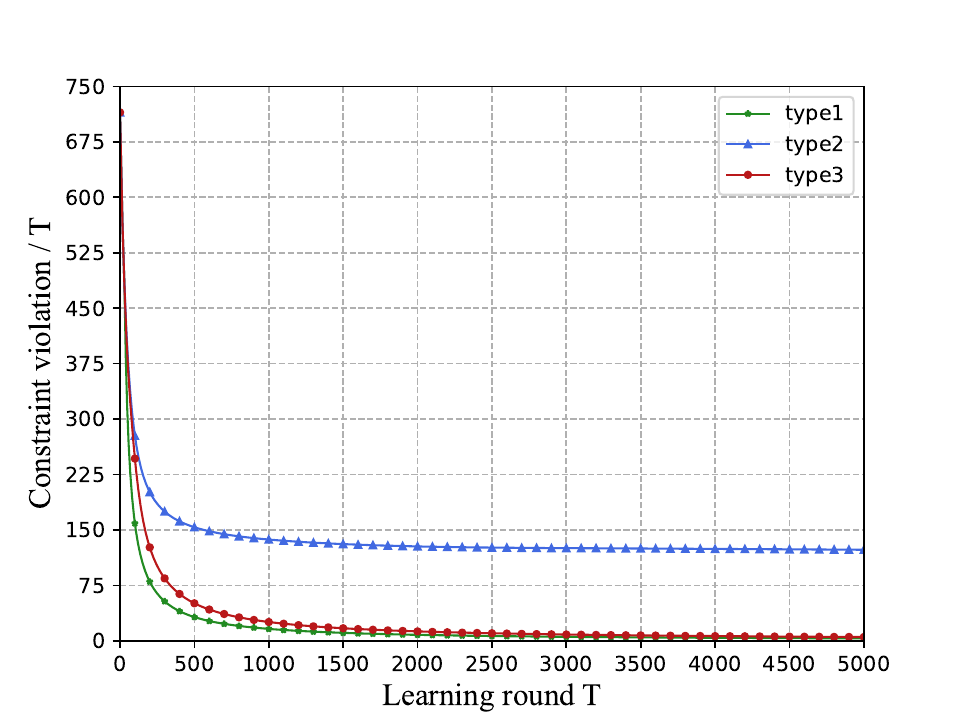}}
\caption{The trajectories of $\mathcal CV(T)/T$ with different types of time-varying delays.}\label{fig_7} 
\end{figure}

\section{Conclusion}\label{sec.five}
In this paper, the distributed online GNE learning problem for multi-cluster games with delayed information feedback and partial-decision information is investigated. A distributed online delay-tolerant GNE learning algorithm is developed. The algorithm is designed within the primal-dual framework, which incorporates an aggregation gradient mechanism, allowing for execution with partial-decision information. To evaluate the performance of the proposed algorithm, metrics such as system-wide regret and constraint violation are formulated. It is demonstrated that these metrics increase sublinearly with respect to time, under some weak conditions. The future work will focus on solving the distributed online GNE learning of multi-cluster game with bandit feedback and time-varying clusters. 
\vspace{1.5ex}
\begin{appendix}
\subsection{Proof of Lemma 1}
By Gerschgorin's disk theorem \cite{Horn_2013} and since the diagonal elements of $W^{-}_{ij}$ are positive, it can be verified that $-1<\lambda\left(W^{-}_{ij}\right)\leq1$. Furthermore, $I_{n-1}-W^{-}_{ij}$ are positive definite by Assumption \ref{assp.four} and Asumption \ref{assp.five}\cite{Hong_2006}. Thus, one has  $|\lambda\left(W^{-}_{ij}\right)|<1$.

Define $e^{ij}_{t}=col\left(e^{ij}_{11,t},\cdots,e^{ij}_{i(j-1),t},e^{ij}_{i(j+1),t},\cdots,e^{ij}_{Nn_{N},t}\right)$ with $e^{ij}_{pq,t}=\bm x^{ij}_{pq,t}-x_{ij,t}$. Similar to Lemma $2$ in \cite{Lu_2021}, it follows from (\ref{eq.A1_2}) that
\begin{align*}
\left\|e^{ij}_{t+1}\right\|&\leq \left\|W^{-}_{ij}\right\|\left\|e^{ij}_{t}\right\|+\left\| (x_{ij,t+1}-x_{ij,t})\bm1_{n-1}\right\|.
\end{align*}

From (\ref{eq.A1_3}), it follows that
\begin{align*}
&\left\| x_{ij,t+1}-x_{ij,t}\right\|
\\&\leq \left\|\mathcal P_{\Omega_{i}}\bigg(\sum_{k=1}^{n_{i}}\left[  W_{i}\right]_{jk}\bigg(x_{ik,t}-\alpha_{t}\sum_{s \in \mathcal S_{ik,t}} y^{s}_{ik,t}\bigg)\bigg)-x_{ij,t}\right\|
\\&\leq \left\|\sum_{k=1}^{n_{i}}\left[  W_{i}\right]_{jk}\bigg(x_{ik,t}-x_{ij,t}-\alpha_{t}\sum_{s \in \mathcal S_{ik,t}} y^{s}_{ik,t}\bigg)\right\|
\\&\leq \sum_{k=1}^{n_{i}}\left\|x_{ik,t}-x_{ij,t}\right\|+G\sum_{k=1}^{n_{i}}\sum_{s \in \mathcal S_{ik,t}}\alpha_{t}(1+\left\|\mu_{ik,t}\right\|),
\end{align*}
where the nonexpansiveness of the Euclidean projection operator and the doubly stochasticity of the mixing matrix $W_{i}$ has been used in deriving the second inequality.
Thus, one has
\begin{align*}
\left\|e^{ij}_{t+1}\right\|&\leq \lambda^{-}\left\|e^{ij}_{t}\right\|+\sqrt{n}\sum_{k=1}^{n_{i}}\left\|x_{ik,t}-x_{ij,t}\right\|
\\&\quad+\sqrt{n}G\sum_{k=1}^{n_{i}}\sum_{s \in \mathcal S_{ik,t}}\alpha_{t}(1+\left\|\mu_{ik,t}\right\|).
\end{align*}

By recursion, one obtains
\begin{align*}
&\left\|e^{ij}_{t+1}\right\|
\\&\leq\left(\lambda^{-}\right)^{t}\left\|e^{ij}_{1}\right\|+\sqrt{n}\sum_{s=0}^{t-1}(\lambda^{-})^{s}\sum_{k=1}^{n_{i}}\left\|x_{ik,t-s}-x_{ij,t-s}\right\|
\\&\quad+\sqrt{n}G\sum_{s=0}^{t-1}(\lambda^{-})^{s}\sum_{k=1}^{n_{i}}\sum_{s\in \mathcal S_{ij,t-s}}\alpha_{t-s}(1+\left\|\mu_{ik,t-s}\right\|).
\end{align*}

Without loss of generality, the bounds of $\bm x^{ij}_{pq,1}$ and $x_{ij,1}$ are assumed to be $R$. Then, it is easy to verify  that $\left\|e^{ij}_{1}\right\|\leq 2\sqrt{n}R$. Thus, for any $i,p \in \mathcal V$, $j \in [n_{i}]$ and $q \in [n_{p}]$, one obtains
\begin{align*}
&\sum_{t=1}^{T}\left\|\bm x^{ij}_{pq,t}-x_{ij,t}\right\|
\\&\leq \sum_{t=1}^{T}\left\|e^{ij}_{t}\right\|
\\&\leq\frac{4\sqrt{n}R}{1-\lambda^{-}}+\frac{\sqrt{n}}{1-\lambda^{-}}\sum_{k=1}^{n_{i}}\sum_{t=1}^{T}\left\|x_{ik,t}-x_{ij,t}\right\|
\\&\quad+\frac{\sqrt{n}G}{1-\lambda^{-}}\sum_{k=1}^{n_{i}}\sum_{t=1}^{T}\sum_{s\in \mathcal S_{ij,t}}\alpha_{t}(1+\left\|\mu_{ik,t}\right\|).
\end{align*}

The proof is complete.

\subsection{Proof of Lemma 2}
If $t=1$, by Algorithm $1$, one has $\mu_{ij,1}=\bm0_{m}$ and $z_{ij,1}=\bm0_{m}$, thus the inequality (\ref{eq.lemm_21}) and (\ref{eq.lemm_22}) hold naturally. Next, it will be shown that equations (19) and (20) hold at time $t>1$ as well by induction. Assume that (\ref{eq.lemm_21}) and (\ref{eq.lemm_22}) hold at time $t$. Then,  by (\ref{eq.A1_8}), (\ref{eq.A1_7}) and the nonexpansive property of the Euclidean projection operator, it follows immediately that
\begin{equation}\label{eq.lemm_24}
\begin{aligned}
\left\|\mu_{ij,t+1}\right\|&\leq\left\|\mu_{ij,t}\right\|+\gamma_{t}\sum_{q=1}^{N}\sum_{k=1}^{n_{q}} [W]^{qk}_{ij}\left(\left\|z_{ij,t}\right\|+\left\|z_{qk,t}\right\|\right)
\\&\quad+\sigma_{t}\gamma_{t}tK.
\end{aligned}
\end{equation}

Substituting (\ref{eq.lemm_21}) and (\ref{eq.lemm_22}) into (\ref{eq.lemm_24}) gives 
\begin{align*}
\left\|\mu_{ij,t+1}\right\|&\leq\frac{\sigma_{t}(t-1)K}{\Delta_{t-1}}+\frac{2\sigma_{t}\gamma_{t}(t-1)K}{\Delta_{t-1}}+\sigma_{t}\gamma_{t}tK
\\&\overset{(a)}{\leq}\frac{((1+2\gamma_{t})\gamma_{t-1}+\gamma_{t})\sigma_{t}tK}{\beta_{t-1}}
\\&\overset{(b)}{\leq}\frac{4\sigma_{t}\gamma_{t-1}tK}{\beta_{t-1}}
\\&\overset{(c)}{\leq}\frac{\sigma_{t+1}tK}{\Delta_{t}},
\end{align*}
where $(a)$ follows from dropping the negative terms; $(b)$ implies from Assumption \ref{assp.seven} that $\gamma_{t}$ and $\beta_{t}$ are nonincreasing; and $(c)$ is derived from Assumption \ref{assp.seven}(b). 

Similarly, by (\ref{eq.A1_6}), one has
\begin{equation}\label{eq.lemm_25}
\begin{aligned}
\left\|z_{ij,t+1}\right\|\leq\left\|z_{ij,t}\right\|+\beta_{t}\sum^{N}_{q=1}\sum^{n_{q}}_{k=1}[W]^{qk}_{ij}(\left\|\mu_{ij,t}\right\|+\left\|\mu_{qk,t}\right\|).
\end{aligned}
\end{equation}

Submitting (\ref{eq.lemm_21}) and (\ref{eq.lemm_22}) into (\ref{eq.lemm_25}) yields
\begin{align*}
\left\|z_{ij,t+1}\right\|&\leq\frac{\sigma_{t}(t-1)K}{\Delta_{t-1}}+2\beta_{t}\frac{\sigma_{t}(t-1)K}{\Delta_{t-1}}
\\&\overset{(a)}{\leq}\frac{(1+2\beta_{1})\sigma_{t}\gamma_{t-1}tK}{\beta_{t-1}}
\\&\overset{(b)}{\leq}\frac{\sigma_{t+1}tK}{\Delta_{t}},
\end{align*}
where $(a)$ results from dropping the negative terms and the nonincreasing property of $\beta_{t}$, and $(b)$ follows from the condition in Assumption \ref{assp.seven}. In summary, the relationships described by (\ref{eq.lemm_21}) and (\ref{eq.lemm_22}) are established.

In the following, it will be shown that $\left\|\mu_{ij,t}-\bar{\mu}_{t}\right\|$  is bounded according to the bounds of $\mu_{ij,t}$ in (\ref{eq.lemm_21}).
First, it follows from (\ref{CAl_5}) that
\begin{equation}\label{eq.lemm_26}
\mu_{t+1}=(W\otimes I_{m})\mu_{t}+\epsilon_{t},
\end{equation}
where $\epsilon_{t}$ is equal to $\left[\mu_{t}+\gamma_{t}\left(\left(L \otimes I_{m}\right)z_{t}-\sigma_{t}{G}_{t}\right)\right]_{+}-\left(W\otimes I_{m}\right)\mu_{t}$.
Furthermore, one has
\begin{equation}\label{eq.lemm_27}
\begin{aligned}
\left\|\epsilon_{t}\right\|&\leq\sqrt{\sum_{s=1}^{n}{\left\|[\epsilon_{t}]_{s}\right\|}^{2}}
\\&\leq\sqrt{n}\left(\left(1+\gamma_{t}\right)\frac{2\sigma_{t}\gamma_{t-1}\left(t-1\right)K}{\beta_{t-1}}+\sigma_{t}\gamma_{t}tK\right)
\\&\leq\frac{5\sqrt{n}\sigma_{t}\gamma_{t-1}tK}{\beta_{t-1}},
\end{aligned}
\end{equation}
where the fact that $\|L\|\le 2$ has been used in deriving the second inequality of (27) by noticing that $W$ is doubly stochastic. 

Combining (\ref{eq.lemm_26}) and (\ref{eq.lemm_27}), one has
\begin{align*}
&\left\|\mu_{t+1}-\bm1_{n}\otimes\bar{\mu}_{t+1}\right\|
\\&\leq\bigg\|((I-\frac{1}{n}\bm1_{n}\bm1^{T}_{n})\otimes I_{m})\epsilon_{t}\bigg\|
+\bigg\|((W-\frac{1}{n}\bm1_{n}\bm1^{T}_{n})\otimes I_{m})
\\&\quad\quad(\mu_{t}-\bm1_{n}\otimes\bar{\mu}_{t})\bigg\|
\\&\leq\lambda\left\|\mu_{t}-\bm1_{n}\otimes\bar{\mu}_{t}\right\|+\frac{5\sqrt{n}\sigma_{t}\gamma_{t-1}tK}{\beta_{t-1}}
\\&\leq\sum_{s=0}^{t-1}\frac{5\sqrt{n}\lambda^{s}\sigma_{t-s}\gamma_{t-s-1}(t-s)K}{\beta_{t-s-1}}.
\end{align*}

The proof is complete.

\subsection{Proof of Lemma 3}
By the iteration of $\mu_{ij,t+1}$ in (\ref{eq.A1_7}), for any $\mu \in \mathbb R^{m}_{\geq 0}$, one obtains 
\begin{align*}
&\left\|\mu_{ij,t+1}-\mu\right\|^{2}
\\&\leq\left\|\mu_{ij,t}-\mu\right\|^{2}+\gamma_{t}^{2}\bigg\|\widetilde{z}_{ij,t}-\sigma_{t}\sum_{s \in \mathcal S_{ij,t}} \left[g^{i,s}_{j}\left(x_{ij,t}\right)\right]_{+}\bigg\|^{2}
\\&\quad+2\gamma_{t}\left<\mu-\mu_{ij,t},\sigma_{t}\sum_{s \in \mathcal S_{ij,t}} \left[g^{i,s}_{j}\left(x_{ij,t}\right)\right]_{+}\right>
\\&\quad+2\gamma_{t}\left<\mu_{ij,t}-\bar{\mu}_{t},\widetilde{z}_{ij,t}\right>+2\gamma_{t}\left<\bar{\mu}_{t}-\mu,\widetilde{z}_{ij,t}\right>
\\&\leq\left\|\mu_{ij,t}-\mu\right\|^{2}+2\sigma_{t}^{2}\gamma_{t}^{2}\left\|\sum_{s \in \mathcal S_{ij,t}} \left[g^{i,s}_{j}\left(x_{ij,t}\right)\right]_{+}\right\|^{2}
\\&\quad+2\gamma_{t}^{2}\bigg\|\widetilde{z}_{ij,t}\bigg\|^{2}+2\sum_{s \in \mathcal S_{ij,t}}\left(\left[g^{i,s}_{j}\left(x_{ij,t}\right)\right]_{+}\right)^\top\mu
\\&\quad+2\gamma_{t}\left\|\mu_{ij,t}-\bar{\mu}_{t}\right\|\left\|\widetilde{z}_{ij,t}\right\|+2\gamma_{t}\left<\bar{\mu}_{t}-\mu,\widetilde{z}_{ij,t}\right>,
\end{align*}
where the last inequality is due to $0<\gamma_{t}\sigma_{t}\leq 1$. 

Summing the above terms over $t=1,\cdots,T$, one obtains 
\begin{equation}\label{eq.lemm_32}
\begin{aligned}
&\sum_{t=1}^{T}\left\|\mu_{ij,t+1}-\mu\right\|^{2}
\\&\leq\sum_{t=1}^{T}\left\|\mu_{ij,t}-\mu\right\|^{2}+2\sum_{t=1}^{T}\sum_{s \in \mathcal S_{ij,t}}\left(\left[g^{i,s}_{j}\left(x_{ij,t}\right)\right]_{+}\right)^\top\mu
\\&\quad+20\sum_{t=2}^{T}\frac{\sqrt{n}\sigma_{t}\gamma_{t-1}^{2}\left(t-1\right)K^{2}}{\beta_{t-1}}\sum_{s=0}^{t-2}\frac{\lambda^{s}\sigma_{t-s-1}\left(t-s-1\right)}{\Delta_{t-s-2}}
\\&\quad+8\sum_{t=2}^{T}\frac{\sigma_{t}^{2}\gamma_{t-1}^{4}\left(t-1\right)^{2}K^{2}}{\beta_{t-1}^{2}}+2\sum_{t=1}^{T}\sigma_{t}^{2}\gamma_{t}^{2}K^{2}\left|\mathcal S_{ij,t}\right|^{2}
\\&\quad+2\sum_{t=1}^{T}\gamma_{t}\left<\bar{\mu}_{t}-\mu,\widetilde{z}_{ij,t}\right>,
\end{aligned}
\end{equation}
where the inequality follows from the nonincreasing property of $\gamma_{t}$, Lemma \ref{lemm_1}, $\left\|\widetilde{z}_{ij,t}\right\|=0$ and $\left\|\mu_{1}-\bm1_{n}\otimes\bar{\mu}_{1}\right\|=0$. Note that $\mu_{ij,1}=\bm0_{m}$, thus 
\begin{equation}\label{eq.lemm_33}
\begin{aligned}
&\frac{1}{2}\sum_{t=1}^{T}\left(\left\|\mu_{ij,t}-\mu\right\|^{2}-\left\|\mu_{ij,t+1}-\mu\right\|^{2}\right)
\\&=\frac{1}{2}\left\|\mu_{ij,1}-\mu\right\|^{2}-\frac{1}{2}\left\|\mu_{ij,T+1}-\mu\right\|^{2}
\\&\leq\frac{1}{2}\left\|\mu\right\|^{2}.
\end{aligned}
\end{equation}

Furthermore, regrouping terms in (\ref{eq.lemm_32}), dividing $2$ on both sides and summing it over $i=1,\cdots,N$, $j=1,\cdots,n_{i}$, and using (\ref{eq.lemm_33}) and the fact that $\sum_{i=1}^{N}\sum_{j=1}^{n_{i}}\left<\bar{\mu}_{t}-\mu,\widetilde{z}_{ij,t}\right>=0$ for $t \in [T]$, the results can be verified. The proof is complete. 
\subsection{Proof of Lemma 4}
Recall from (\ref{CAl_2}) that
\begin{equation*}
x_{t+1}=\mathcal P_{\bm\Omega}\left(\widehat W\left(x_{t}-\alpha_{t}Y_{t}\right)\right),
\end{equation*}
where $Y_{t}=\widetilde F_{t}+\widehat \mu_{t} \widetilde G_{t}$. To simplify the notation, let $\widehat{x}_{ij,t+1}=\sum_{k=1}^{n_{i}}\left[W_{i}\right]_{jk}\left(x_{ik,t}-\alpha_{t}\sum_{s \in \mathcal S_{ik,t}} y^{s}_{ik,t}\right)$, $\theta_{ij,t+1}=x_{ij,t+1}-\widehat {x}_{ij,t+1}$ and $\Theta_{t+1}=col\left(\theta_{ij,t+1}\right)$. Then, the above recursion can be written as
\begin{equation}\label{eq.lemm_41}
x_{t+1}=\widehat W\left(x_{t}-\alpha_{t}Y_{t}\right)+\Theta_{t+1}.
\end{equation}

Using (\ref{eq.lemm_41}) repeatedly, it can be shown that, for $t\geq2$,
\begin{equation}\label{eq.lemm_42}
\begin{aligned}
x_{t}=&\widehat W^{t-1} x_{1}-\sum_{v=0}^{t-2}\alpha_{t-1-v}\widehat W^{v+1}Y_{t-1-v}
+\sum^{t-2}_{v=0}\widehat W^{v}\Theta_{t-v}.
\end{aligned}
\end{equation}

According to (\ref{eq.lemm_42}), the iteration of agent $j$ in cluster $i$ can be derived as
\begin{equation}\label{eq.lemm_43}
\begin{aligned}
x_{ij,t}=&\sum_{k=1}^{n_{i}}\left[W_{i}^{t-1}\right]_{jk}x_{ik,1}+\sum_{v=0}^{t-2}\sum_{k=1}^{n_{i}}\left[W_{i}^{v}\right]_{jk}\theta_{ik,t-v}
\\&-\sum_{v=0}^{t-2}\alpha_{t-1-v}\sum_{k=1}^{n_{i}}\left[W_{i}^{v+1}\right]_{jk}\sum_{s \in \mathcal S_{ik,t-1-v}}y^{s}_{ik,t-1-v}.
\end{aligned}
\end{equation}

Similarly, multipling both sides of (\ref{eq.lemm_42}) by the diagonal matrix $D=diag\left(\frac{1}{n_{i}}\bm 1^\top_{n_{i}}\right)_{i \in \mathcal V}$, it yields that, for $i \in \mathcal V$,
\begin{equation}\label{eq.lemm_44}
\begin{aligned}
\bar{x}_{i,t}=\bar{x}_{i,1}-\sum_{v=0}^{t-2}\alpha_{t-1-v}\bar{y}_{i,t-1-v}+\sum_{v=0}^{t-2}\bar{\theta}_{i,t-v},
\end{aligned}
\end{equation}
where $\bar{x}_{i,t}=\frac{1}{n_{i}}\sum_{k=1}^{n_{i}}x_{ik,t}$, $\bar{\theta}_{i,t-v}=\frac{1}{n_{i}}\sum_{k=1}^{n_{i}}\theta_{ik,t-v}$ and $\bar{y}_{i,t-1-v}=\frac{1}{n_{i}}\sum_{k=1}^{n_{i}}\sum_{s \in \mathcal S_{ik,t-1-v}}y^{s}_{ik,t-1-v}$. Hence, by combining (\ref{eq.lemm_43}) and (\ref{eq.lemm_44}), one obtains 
\begin{equation}\label{eq.lemm_45}
\begin{aligned}
&\left\|x_{ij,t}-\bar{x}_{i,t}\right\|
\\&\leq\sum_{k=1}^{n_{i}}\Delta W_{ij,k}(t-1)\left\|x_{ik,1}\right\|+\sum_{v=0}^{t-2}\sum_{k=1}^{n_{i}}\Delta W_{ij,k}(v)\left\|\theta_{ik,t-v}\right\|
\\&\quad+\sum_{v=0}^{t-2}\alpha_{t-1-v}\sum_{k=1}^{n_{i}}\Delta W_{ij,k}(v+1)\sum_{s \in \mathcal S_{ik,t-1-v}}\left\|y^{s}_{ik,t-1-v}\right\|,
\end{aligned}
\end{equation}
where $\Delta W_{ij,k}(t)=\left|\left[W_{i}^{t}\right]_{jk}-\frac{1}{n_{i}}\right|$.
Moreover, by the definition of $\theta_{ij,t+1}$, it follows that
\begin{equation}\label{eq.lemm_46}
\begin{aligned}
\sum_{j=1}^{n_{i}}\left\|\theta_{ij,t+1}\right\|
&=\sum_{j=1}^{n_{i}}\left\|x_{ij,t+1}-\sum_{k=1}^{n_{i}}\left[W_{i}\right]_{jk}x_{ik,t}\right\|
\\&\quad+\alpha_{t}\sum_{j=1}^{n_{i}}\sum_{k=1}^{n_{i}}\left[W_{i}\right]_{jk}\sum_{s \in \mathcal S_{ik,t}}\left\|y^{s}_{ik,t}\right\|
\\&\leq 2\alpha_{t}\sum_{j=1}^{n_{i}}\sum_{k=1}^{n_{i}}\left[W_{i}\right]_{jk}\sum_{s \in \mathcal S_{ik,t}}\left\|y^{s}_{ik,t}\right\|
\\&=2\alpha_{t}\sum_{k=1}^{n_{i}}\sum_{s \in \mathcal S_{ik,t}}\left\|y^{s}_{ik,t}\right\|.
\end{aligned}
\end{equation}

Combining the above relation with (\ref{eq.lemm_45}) and Lemma \ref{lemm_1}, one obtains
\begin{align*}
&\left\|x_{ij,t}-\bar{x}_{i,t}\right\|
\\&\leq\sum_{k=1}^{n_{i}}\psi^{t-4}\left\|x_{ik,1}\right\|+\sum_{v=0}^{t-2}\psi^{v-3}\sum_{k=1}^{n_{i}}\left\|\theta_{ik,t-v}\right\|
\\&\quad+\sum_{v=0}^{t-2}\alpha_{t-1-v}\psi^{v-2}\sum_{k=1}^{n_{i}}\sum_{s \in \mathcal S_{ik,t-1-v}}\left\|y^{s}_{ik,t-1-v}\right\|
\\&\leq n_{i}R\psi^{t-4}+3G\sum_{v=0}^{t-2}\alpha_{t-1-v}\psi^{v-3}\sum_{k=1}^{n_{i}}\left|\mathcal S_{ik,t-1-v}\right|\times
\\&\quad \left(1+\left\|\mu_{ik,t-1-v}\right\|\right)
\\&\leq nR\psi^{t-4}+3ncG\sum_{v=0}^{t-2}\psi^{v-3}(T^{-a_{1}}+KT^{1+a_{2}-a_{1}-a_{3}}),
\end{align*}
where $\psi \triangleq 1-\frac{a}{4n_{i}^{2}}< 1$, the first inequality is derived from Proposition 1 in \cite{Nedic_2010} and the last inequality is obtained by setting $\alpha_{t}$, $\beta_{t}$ and $\gamma_{t}$ as in Lemma \ref{lemma_6}. The proof is complete.

\subsection{Proof of Lemma 5}
Adding and subtracting two terms $\bar{x}_{i,t}$ and $x^{*}_{i,t}$ in $\left\|\bar{x}_{i,t+1}-x^{*}_{i,t+1}\right\|$ yields that 
\begin{equation}\label{lemm_31_1}
\begin{aligned}
\left\|\bar{x}_{i,t+1}-x^{*}_{i,t+1}\right\|^{2}&=\left\|\bar{x}_{i,t}-x^{*}_{i,t}\right\|^{2}-\left\|\bar{x}_{i,t+1}-\bar{x}_{i,t}\right\|^{2}
\\&\quad+\left<2\bar{x}_{i,t+1}-x^{*}_{i,t}-x^{*}_{i,t+1},x^{*}_{i,t}-x^{*}_{i,t+1}\right>
\\&\quad+2\left<\bar{x}_{i,t+1}-\bar{x}_{i,t},\bar{x}_{i,t+1}-x^{*}_{i,t}\right>.
\end{aligned}
\end{equation}

Since $\bar{x}_{i,t},x^{*}_{i,t} \in \Omega_{i}$ for $t \in [T]$, the third term on the right-hand side of (\ref{lemm_31_1}) can be bounded as
\begin{equation}\label{lemm_31_7}
\begin{aligned}
\left<2\bar{x}_{i,t+1}-x^{*}_{i,t}-x^{*}_{i,t+1},x^{*}_{i,t}-x^{*}_{i,t+1}\right>\leq 4R\left\|x^{*}_{i,t}-x^{*}_{i,t+1}\right\|.
\end{aligned}
\end{equation}

For the fourth term on the right-hand side of (\ref{lemm_31_1}), one has
\begin{equation}\label{lemm_31_2}
\begin{aligned}
&\left<\bar{x}_{i,t+1}-\bar{x}_{i,t},\bar{x}_{i,t+1}-x^{*}_{i,t}\right>
\\&=\left<-\alpha_{t}\bar{y}_{i,t}+\bar{\theta}_{i,t+1},\bar{x}_{i,t+1}-x^{*}_{i,t}\right>
\\&=\left<-\alpha_{t}\bar{y}_{i,t},\bar{x}_{i,t+1}-x_{ij,t+1}\right>
\\& \quad +\left<\bar{\theta}_{i,t+1},\bar{x}_{i,t+1}-x_{ij,t+1}\right>
\\&\quad+\left<-\alpha_{t}\bar{y}_{i,t}+\bar{\theta}_{i,t+1},x_{ij,t+1}-x^{*}_{i,t}\right>,
\end{aligned}
\end{equation}
where the first equality is derived by multiplying both sides of (\ref{eq.lemm_41}) with the diagonal matrix $D=diag\left(\frac{1}{n_{i}}\bm 1^\top_{n_{i}}\right)_{i \in \mathcal V}$.

For the last term on the right-hand side of (\ref{lemm_31_2}), it can be verified that
\begin{align*}
&\left<-\alpha_{t}\bar{y}_{i,t}+\bar{\theta}_{i,t+1},x_{ij,t+1}-x^{*}_{i,t}\right>
\\& \leq \alpha_{t}\left<\bar{y}_{i,t},x^{*}_{i,t}-x_{ij,t+1}\right>+\left<\bar{\theta}_{i,t+1},x_{ij,t+1}-\bar x_{i,t+1}\right>
\\&\quad+\frac{1}{n_{i}}\sum_{k=1}^{n_{i}}\left<{\theta}_{ik,t+1},\bar x_{i,t+1}-x_{ik,t+1}+x_{ik,t+1}-x^{*}_{i,t}\right>.
\end{align*}

By the fact that
\begin{equation*}
\left<\theta_{ik,t+1},x_{ik,t+1}-x^{*}_{i,t}\right>\leq0,
\end{equation*}
which always holds for $i \in \mathcal V$, $k \in [n_{i}]$ and $t \in \left[T\right]$ due to the definition of $\theta_{ik,t+1}$ and the property of the projection operator $(\mathcal {P}_{\chi}\left(x\right)-x)^\top({\mathcal P}_{\chi}\left(x\right)-y)\leq 0$ for all $x \in \mathbb R$ and $y \in \chi$, one can obtain
\begin{equation}\label{lemm_31_8}
\begin{aligned}
&\left<-\alpha_{t}\bar{y}_{i,t}+\bar{\theta}_{i,t+1},x_{ij,t+1}-x^{*}_{i,t}\right>
\\& \leq \alpha_{t}\left<\bar{y}_{i,t},x^{*}_{i,t}-\bar{x}_{i,t}\right>+\alpha_{t}\left<\bar{y}_{i,t},\bar{x}_{i,t}-\bar{x}_{i,t+1}\right>
\\&\quad+\alpha_{t}\left<\bar{y}_{i,t},\bar{x}_{i,t+1}-x_{ij,t+1}\right>+\left<\bar{\theta}_{i,t+1},x_{ij,t+1}-\bar x_{i,t+1}\right>
\\&\quad+\frac{1}{n_{i}}\sum_{k=1}^{n_{i}}\left<{\theta}_{ik,t+1},\bar x_{i,t+1}-x_{ik,t+1}\right>.
\end{aligned}
\end{equation}

Substituting (\ref{lemm_31_7}), (\ref{lemm_31_2}), (\ref{lemm_31_8}) into (\ref{lemm_31_1}), multiplying both sides by $\frac{n_{i}}{2\alpha_{t}}$, summing it from $t=1,\cdots,T$ and $i=1,\cdots,N$, and regrouping the terms, one obtains
\begin{equation}\label{lemm_31_50}
\begin{aligned}
&\sum_{t=1}^{T}\sum_{i=1}^{N}\frac{n_{i}}{2\alpha_{t}}\left\|\bar x_{i,t+1}-\bar x_{i,t}\right\|^{2}
\\&\leq\sum_{t=1}^{T}\sum_{i=1}^{N}\frac{n_{i}}{2\alpha_{t}}\left(\left\|\bar{x}_{i,t}-x^{*}_{i,t}\right\|^{2}-\left\|\bar{x}_{i,t+1}-x^{*}_{i,t+1}\right\|^{2}\right)
\\&\quad+\frac{2nR}{\alpha_{T}}\Phi_{T}+\sum_{t=1}^{T}\sum_{i=1}^{N}\sum_{k=1}^{n_{i}}\left<\sum_{s \in \mathcal S_{ik,t}}y^{s}_{ik,t},x^{*}_{i,t}-\bar x_{i,t}\right>
\\&\quad+\sum_{t=1}^{T}\sum_{i=1}^{N}\sum_{k=1}^{n_{i}}\left<\sum_{s \in \mathcal S_{ik,t}}y^{s}_{ik,t},\bar x_{i,t}-\bar x_{i,t+1}\right>
\\&\quad+\sum_{t=1}^{T}\frac{1}{\alpha_{t}}\sum_{i=1}^{N}\sum_{k=1}^{n_{i}}\left<{\theta}_{ik,t+1},\bar x_{i,t+1}-x_{ik,t+1}\right>.
\end{aligned}
\end{equation}

To simplify the notation, denote $\nabla f^{i,t}_{k}\left(\cdot\right)\triangleq \nabla_{x_{ik,t}} f^{i,t}_{k}\left(\cdot\right)$, $\sum_{t,i,k}(\cdot)\triangleq \sum_{t=1}^{T}\sum_{i=1}^{N}\sum_{k=1}^{n_{i}}(\cdot)$ and  $\sum_{i,k}(\cdot)\triangleq\sum_{i=1}^{N}\sum_{k=1}^{n_{i}}(\cdot)$.

Next, consider the upper bounds of the last three terms on the right-hand side of (\ref{lemm_31_50}). Firstly, the term $\sum_{t,i,k}\left<\sum_{s \in \mathcal S_{ik,t}}y^{s}_{ik,t},x^{*}_{i,t}-\bar x_{i,t}\right>$ will be bounded by the optimal conditions. For agent $k$ in cluster $i$, by the KKT condition in (\ref{eq.KKT_2}), there exists  $n_{ik,t} \in N_{\Omega_{i}}\left(x_{i,t}^{*}\right)$ such that $0=\nabla f^{i,t}_{k}\left(x_{i,t}^{*},x_{-i,t}^{*}\right)+\nabla g^{i,t}_{k}\left(x_{i,t}^{*}\right)^\top\mu^{*}_{t}+n_{ik,t}+L_{i}[k,:]\lambda^{*}_{i,t}$,
where $L_{i}[k,:]$ represents the $k$-th row of matrix $L_{i}$. Consequently, one has 
\begin{equation}\label{lemm_31_10}
\begin{aligned}
&\sum_{t,i,k}\bigg<\sum_{s \in \mathcal S_{ik,t}}y^{s}_{ik,t},x^{*}_{i,t}-\bar x_{i,t}\bigg>
\\&\leq\sum_{t,i,k}\bigg<\sum_{s \in \mathcal S_{ik,t}}\nabla f^{i,s}_{k}(x_{ik,t},\bm x^{-i}_{ik,t})-\nabla f^{i,t}_{k}\left(\bar x_{t}\right),x^{*}_{i,t}-\bar x_{i,t}\bigg>
\\& \quad+\sum_{t,i,k}\bigg<\sum_{s \in \mathcal S_{ik,t}}(\mu_{ik,t})^\top\nabla \left[g^{i,s}_{k}(x_{ik,t})\right]_{+},x^{*}_{i,t}-\bar x_{i,t}\bigg>
\\&\quad -\sum_{t,i,k}\left<\nabla g^{i,t}_{k}\left(x_{i,t}^{*}\right)^\top\mu^{*}_{t},x^{*}_{i,t}-\bar x_{i,t}\right>
\\&\quad-\sum_{t,i,k}\left<n_{ik,t}+L_{i}[k,:]\lambda^{*}_{i,t},x^{*}_{i,t}-\bar x_{i,t}\right>
\\&\quad -\sum_{t,i,k}\left<\nabla f^{i,t}_{k}\left( x^{*}_{t}\right)-\nabla f^{i,t}_{k}\left(\bar x_{t}\right),x^{*}_{i,t}-\bar x_{i,t}\right>,
\end{aligned}
\end{equation}
where $\nabla f^{i,t}_{k}\left( x^{*}_{t}\right)\triangleq\nabla f^{i,t}_{k}\left(x^{*}_{i,t},x^{*}_{-i,t}\right)$, $\nabla f^{i,t}_{k}\left(\bar x_{t}\right)\triangleq\nabla f^{i,t}_{k}\left( \bar x_{i,t}, \bar x_{-i,t}\right)$ and $\bar x_{-i,t}=col(\bm 1_{n_{1}} \otimes \bar x_{1,t},\cdots,\bm 1_{n_{i-1}} \otimes \bar x_{i-1,t},\bm 1_{n_{i+1}} \otimes \bar x_{i+1,t},\cdots,\bm 1_{n_{N}} \otimes \bar x_{N,t})$.

For the first term on the right-hand side of (\ref{lemm_31_10}), one has
\begin{equation}\label{lemm_31_80}
\begin{aligned}
&\sum_{t,i,k}\bigg<\sum_{s \in \mathcal S_{ik,t}}\nabla f^{i,s}_{k}(x_{ik,t},\bm x^{-i}_{ik,t})-\nabla f^{i,t}_{k}\left(\bar x_{t}\right),x^{*}_{i,t}-\bar x_{i,t}\bigg>
\\&=\sum_{t,i,k}\bigg<\sum_{s \in \mathcal S_{ik,t}}(\nabla f^{i,s}_{k}(x_{ik,t},\bm x^{-i}_{ik,t})-\nabla f^{i,s}_{k}(\bar x_{s})),x^{*}_{i,t}-\bar x_{i,t}\bigg>
\\&\quad +\sum_{t,i,k}\bigg<\sum_{s \in \mathcal S_{ik,t}}\nabla f^{i,s}_{k}(\bar x_{s})-\nabla f^{i,t}_{k}\left(\bar x_{t}\right),x^{*}_{i,t}-\bar x_{i,t}\bigg>.
\end{aligned}
\end{equation}

By the boundness of $x^{*}_{i,t}$ and $\bar x_{i,t}$ and Assumption \ref{assp.three}, the first term on the right-hand side of (\ref{lemm_31_80}) can be bounded as
\begin{align*}
&\sum_{t,i,k}\bigg<\sum_{s \in \mathcal S_{ik,t}}(\nabla f^{i,s}_{k}(x_{ik,t},\bm x^{-i}_{ik,t})-\nabla f^{i,s}_{k}(\bar x_{s})),x^{*}_{i,t}-\bar x_{i,t}\bigg>
\\&\leq 2R\sum_{t,i,k}\sum_{s \in \mathcal S_{ik,t}}\Big\|\nabla f^{i,s}_{k}(x_{ik,t},\bm x^{-i}_{ik,t})-\nabla f^{i,s}_{k}\left(\bar x_{s}\right)\Big\|
\\&\leq 2R\Big(\sum_{t,i,k}\sum_{s \in \mathcal S_{ik,t}}\Big\|\nabla f^{i,s}_{k}(x_{ik,t},\bm x^{-i}_{ik,t})-\nabla f^{i,s}_{k}(x_{ik,s},\bm x^{-i}_{ik,s})\Big\|
\\&\quad +\sum_{t,i,k}\sum_{s \in \mathcal S_{ik,t}}\Big\|\nabla f^{i,s}_{k}(x_{ik,s},\bm x^{-i}_{ik,s})-\nabla f^{i,s}_{k}\left(\bar x_{s}\right)\Big\|\Big)
\\&\leq 2Rl\sum_{t,i,k}\sum_{s \in \mathcal S_{ik,t}}\left(\left\|x_{ik,t}-x_{ik,s}\right\|+\left\|\bm x^{-i}_{ik,t}-\bm x^{-i}_{ik,s}\right\|\right)
\\&\quad +2R\sum_{t,i,k}\sum_{s \in \mathcal S_{ik,t}}\left\|\nabla f^{i,s}_{k}(x_{ik,s},\bm x^{-i}_{ik,s})-\nabla f^{i,s}_{k}(\bar x_{i,s},\bm x^{-i}_{ik,s})\right\|
\\&\quad +2R\sum_{t,i,k}\sum_{s \in \mathcal S_{ik,t}}\left\|\nabla f^{i,s}_{k}(\bar x_{i,s},\bm x^{-i}_{ik,s})-\nabla f^{i,s}_{k}(\bar x_{i,s},x_{-i,s})\right\|
\\& \quad +2R\sum_{t,i,k}\sum_{s \in \mathcal S_{ik,t}}\left\|\nabla f^{i,s}_{k}(\bar x_{i,s},x_{-i,s})-\nabla f^{i,s}_{k}(\bar x_{s})\right\|.
\end{align*}

Thus, one has
\begin{equation}\label{lemm_31_81}
\begin{aligned}
&\sum_{t,i,k}\bigg<\sum_{s \in \mathcal S_{ik,t}}(\nabla f^{i,s}_{k}(x_{ik,t},\bm x^{-i}_{ik,t})-\nabla f^{i,s}_{k}(\bar x_{s})),x^{*}_{i,t}-\bar x_{i,t}\bigg>
\\&\leq 2lR(2+c)\sum_{t,i,k}\left(\left\|x_{ik,t}-\bar x_{i,t}\right\|+\left\|x_{-i,t}-\bar x_{-i,t}\right\|\right)
\\&\quad+2lR(2+c)\sum_{t,i,k}\left\|\bm x^{-i}_{ik,t}-x_{-i,t}\right\|
\\&\quad+2lR\sum_{t,i,k}\sum_{s \in \mathcal S_{ik,t}\setminus\left\{t\right\}}\left\|\bar x_{i,t}-\bar x_{i,s}\right\|
\\&\quad +2lR\sum_{t,i,k}\sum_{s \in \mathcal S_{ik,t}\setminus\left\{t\right\}}\left\|\bar x_{-i,t}-\bar x_{-i,s}\right\|.
\end{aligned}
\end{equation}

For the second term on the right-hand side of (\ref{lemm_31_80}), by the definition of $\mathcal U_{ik,t}$ and $\mathcal S_{ik,t}$, it can be verified that
\begin{align*}
&\sum_{t,i,k}\bigg<\sum_{s \in \mathcal S_{ik,t}}\nabla f^{i,s}_{k}(\bar x_{s})-\nabla f^{i,t}_{k}\left(\bar x_{t}\right),x^{*}_{i,t}-\bar x_{i,t}\bigg>
\\&=\sum_{i,k}\sum_{t \in \mathcal U_{ik,T}}\bigg<\nabla f^{i,t}_{k}\left(\bar x_{t}\right),\bar x_{i,t}-x^{*}_{i,t}\bigg>
\\&\quad+\sum_{t,i,k}\sum_{s \in \mathcal S_{ik,t}\setminus\left\{t\right\}}\bigg<\nabla f^{i,s}_{k}(\bar x_{s}),x^{*}_{i,t}-\bar x_{i,t}+\bar x_{i,s}-x^{*}_{i,s}\bigg>
\\&\leq2R\sum_{i,k}\sum_{t \in \mathcal U_{ik,T}}\left\|\nabla f^{i,t}_{k}\left(\bar x_{t}\right)\right\|+G\sum_{t,i,k}\sum_{s \in \mathcal S_{ik,t}\setminus\left\{t\right\}}\left\|\bar x_{i,t}-\bar x_{i,s}\right\|
\\&\quad +G\sum_{t,i,k}\sum_{s \in \mathcal S_{ik,t}\setminus\left\{t\right\}}\sum_{h=s}^{t-1}\left\|x^{*}_{i,h+1}-x^{*}_{i,h}\right\|.
\end{align*}

Next, for $i \in \mathcal V$, $k \in [n_{i}]$, it holds that
\begin{align*}
&\sum_{t=1}^{T}\sum_{s \in \mathcal S_{ik,t}\setminus\left\{t\right\}}\sum_{h=s}^{t-1}\left\|x^{*}_{i,h+1}-x^{*}_{i,h}\right\|
\\&= \sum_{t=1}^{T}\left|[t]\cap(\cup_{h=t+1}^{T}\mathcal S_{ik,h})\right|\left\|x^{*}_{i,t+1}-x^{*}_{i,t}\right\|
\\&\leq\sum_{t=1}^{T}\left|[t]\cap([T]\setminus\cup_{h=1}^{t}\mathcal S_{ik,h})\right|\left\|x^{*}_{i,t+1}-x^{*}_{i,t}\right\|
\\&\leq \left|\mathcal U_{ik,[T]}\right|\sum_{t=1}^{T}\left\|x^{*}_{i,t+1}-x^{*}_{i,t}\right\|,
\end{align*}
where the first inequality follows from $\cup_{h=1}^{T}\mathcal S_{ik,h}\subseteq [T]$ and the second inequality is derived from the fact that $[t]\cap([T]\setminus\cup_{h=1}^{t}\mathcal S_{ik,h})=\mathcal U_{ik,t}$.

Therefore,
\begin{equation}\label{lemm_31_82}
\begin{aligned}
&\sum_{t,i,k}\bigg<\sum_{s \in \mathcal S_{ik,t}}\nabla f^{i,s}_{k}(\bar x_{s})-\nabla f^{i,t}_{k}\left(\bar x_{t}\right),x^{*}_{i,t}-\bar x_{i,t}\bigg>
\\&\leq2RG\sum_{i,k}\left|\mathcal U_{ik,T}\right|+G\sum_{t,i,k}\sum_{s \in \mathcal S_{ik,t}\setminus\left\{t\right\}}\left\|\bar x_{i,t}-\bar x_{i,s}\right\|
\\&\quad+nG\mathop{\max}\limits_{i \in \mathcal V, k \in [n_{i}]}\left|\mathcal U_{ik,[T]}\right|\Phi_{T}.
\end{aligned}
\end{equation}

Thus, substituting (\ref{lemm_31_81}) and (\ref{lemm_31_82}) into (\ref{lemm_31_80}) yields
\begin{align*}
&\sum_{t,i,k}\bigg<\sum_{s \in \mathcal S_{ik,t}}\nabla f^{i,s}_{k}(x_{ik,t},\bm x^{-i}_{ik,t})-\nabla f^{i,t}_{k}\left(\bar x_{t}\right),x^{*}_{i,t}-\bar x_{i,t}\bigg>
\\&\leq 2lR(2+c)\sum_{t,i,k}\big(\left\|x_{ik,t}-\bar x_{i,t}\right\|+\left\|x_{-i,t}-\bar x_{-i,t}\right\|
\\&\quad+\left\|\bm x^{-i}_{ik,t}-x_{-i,t}\right\|\big)+2lR\sum_{t,i,k}\sum_{s \in \mathcal S_{ik,t}\setminus\left\{t\right\}}\left\|\bar x_{-i,t}-\bar x_{-i,s}\right\|
\\&\quad +(2lR+G)\sum_{t,i,k}\sum_{s \in \mathcal S_{ik,t}\setminus\left\{t\right\}}\left\|\bar x_{i,t}-\bar x_{i,s}\right\|
\\&\quad+2RG\sum_{i,k}\left|\mathcal U_{ik,T}\right|+nG\mathop{\max}\limits_{i \in \mathcal V, k \in [n_{i}]}\left|\mathcal U_{ik,[T]}\right|\Phi_{T}.
\end{align*}

From (\ref{eq.lemm_41}), it follows that $\bar x_{i,t}=\bar x_{i,t-1}-\alpha_{t-1}\bar y_{i,t-1}+\bar \theta_{i,t}$ for $t \geq 2$. Thus, for $t>s$, one has
\begin{align*}
\bar x_{i,t}=\bar x_{i,s}-\sum_{k=s+1}^{t}(\alpha_{k-1}\bar y_{i,k-1}-\bar \theta_{i,k}).
\end{align*}

Combining the above relation with (\ref{eq.lemm_46}), it can be easily verified that
\begin{align*}
&\sum_{t,i,k}\sum_{s \in \mathcal S_{ik,t}\setminus\left\{t\right\}}\left\|\bar x_{i,t}-\bar x_{i,s}\right\|
\\&\leq \sum_{t,i,k}\sum_{s \in \mathcal S_{ik,t}\setminus\left\{t\right\}}\sum_{k=s+1}^{t}\left(\alpha_{k-1}\left\|\bar y_{i,k-1}\right\|+\left\|\bar \theta_{i,k}\right\|\right)
\\& \leq 3cG\sum_{t,i,k}\sum_{s \in \mathcal S_{ik,t}\setminus\left\{t\right\}}\sum_{k=s+1}^{t}\frac{\alpha_{k-1}}{n_{i}}\sum_{j=1}^{n_{i}}(1+\left\|\mu_{ij,k-1}\right\|)
\\& \leq 3cG\sum_{t,i,k}\sum_{s \in \mathcal S_{ik,t}\setminus\left\{t\right\}}(t-s)(T^{-a_{1}}+KT^{1+a_{2}-a_{1}-a_{3}})
\\& < 3cG(T^{-a_{1}}+KT^{1+a_{2}-a_{1}-a_{3}})\sum_{t,i,k}\sum_{s \in \mathcal S_{ik,t}\setminus\left\{t\right\}}\tau_{ik,s},
\end{align*}
where the last inequality follows from the definition of $\mathcal S_{ij,t}$. 

By Assumption \ref{assp.ten}, one obtains
\begin{align*}
&\sum_{t,i,k}\sum_{s \in \mathcal S_{ik,t}\setminus\left\{t\right\}}\left\|\bar x_{i,t}-\bar x_{i,s}\right\|
\\& < 3cnG(T^{\tau-a_{1}}+KT^{1+\tau+a_{2}-a_{1}-a_{3}}).
\end{align*}

Following this, by Lemma \ref{lemma_2} and Lemma \ref{lemma_6}, one has
\begin{align*}
&\sum_{t,i,k}\left\|\bm x^{-i}_{ik,t}-x_{-i,t}\right\|
\\&\leq \sum_{i,k}\sum_{p\in\mathcal V\setminus \left\{i\right\}}\sum_{q=1}^{n_{p}}\sum_{t=1}^{T}\left\|\bm x_{ik,t}^{pq}-x_{pq,t}\right\|
\\&\leq \frac{n^{\frac{5}{2}}}{1-\lambda^{-}}(4R+2nb_{1})+\frac{n^{\frac{7}{2}}}{1-\lambda^{-}}(2b_{2}+cG)T^{1-a_{1}}
\\&\quad+\frac{n^{\frac{7}{2}}}{1-\lambda^{-}}(2b_{3}+cKG)T^{2+a_{2}-a_{1}-a_{3}}.
\end{align*}

Summarizing, one obtains 
\begin{equation}\label{lemm_31_14}
\begin{aligned}
&\sum_{t,i,k}\bigg<\sum_{s \in \mathcal S_{ik,t}}\nabla f^{i,s}_{k}(x_{ik,t},\bm x^{-i}_{ik,t})-f^{i,t}_{k}\left(\bar x_{t}\right),x^{*}_{i,t}-\bar x_{i,t}\bigg>
\\& \leq b_{4}+b_{5}T^{\tau-a_{1}}+b_{6}T^{1+\tau+a_{2}-a_{1}-a_{3}}+nG(2R+1)T^{u+\phi},
\end{aligned}
\end{equation}
where $b_{4}=2nlR(2+c)(b_{1}(1+n)+\frac{n^{\frac{3}{2}}}{1-\lambda^{-}}(4R+2nb_{1}))$, $b_{5}=2nlb_{2}R(2+c)(1+n)+2lR(2b_{2}+cG)(2+c)\frac{n^{\frac{7}{2}}}{1-\lambda^{-}}+6cnlRG(1+n)+3cnG^{2}$ and $b_{6}=2nlb_{3}R(2+c)(1+n)+2lR(2b_{3}+cKG)(2+c)\frac{n^{\frac{7}{2}}}{1-\lambda^{-}}+6cnlKRG(1+n)+3cnG^{2}K$.

For the second term and third term on the right-hand side of (\ref{lemm_31_10}), by the definition of $\nabla \left[g^{i,s}_{k}(x_{ik,t})\right]_{+}$ and the boundedness of $\left\|\nabla g^{i,s}_{k}(x_{ik,t})\right\|$, one obtains
\begin{equation}\label{lemm_31_11}
\begin{aligned}
& \sum_{t,i,k}\bigg<\sum_{s \in \mathcal S_{ik,t}}(\mu_{ik,t})^\top\nabla \left[g^{i,s}_{k}(x_{ik,t})\right]_{+},x^{*}_{i,t}-\bar x_{i,t}\bigg>
\\&-\sum_{t,i,k}\left<\nabla g^{i,t}_{k}\left(x_{i,t}^{*}\right)^\top\mu^{*}_{t},x^{*}_{i,t}-\bar x_{i,t}\right>
\\& \leq 2cRG\sum_{t,i,k}\left\|\mu_{ik,t}\right\|+\sum_{t,i,k}(g^{i,t}_{k}\left(\bar x_{i,t}\right)-g^{i,t}_{k}\left(x^{*}_{i,t}\right))^\top\mu^{*}_{t}
\\&\leq 2cRG\sum_{t,i,k}\left\|\mu_{ik,t}\right\|+\sum_{t,i,k}(g^{i,t}_{k}\left(\bar x_{i,t}\right)-g^{i,t}_{k}\left(x_{ik,t}\right))^\top\mu^{*}_{t}
\\&\quad +\sum_{t,i,k}\left(\left[g^{i,t}_{k}(x_{ik,t})\right]_{+}\right)^\top\mu^{*}_{t}
\\& \leq (2cKR+\vartheta b_{3})nGT^{2+a_{2}-a_{3}}+\vartheta nG(b_{1}+b_{2}T^{1-a_{1}})
\\&\quad +\vartheta\mathcal CV(T),
\end{aligned}
\end{equation}
where the first inequality follows from the convexity of $g^{i,t}_{k}$ and the last inequality results from Lemma $1$ in \cite{Nedic_2009}, in which there exists a constant $\vartheta>0$, such that $\left\|\mu_{t}^{*}\right\|\leq \vartheta$ holds for $t \in [T]$.

For the fourth term on the right-hand side of (\ref{lemm_31_10}), according to the definition of normal cone, one has that $n_{ik,t}\in N_{\Omega_{i}}(x^{*}_{i,t})$ if and only if $\left<n_{ik,t},x^{*}_{i,t}-y\right>\geq 0$ for all $y \in \Omega_{i}$. For any $i \in \mathcal V$, because $\bar x_{i,t}$ always lies in $\Omega_{i}$ and the sum of all elements in each row of matrix $L_{i}$ is $0$, it follows that
\begin{equation}\label{lemm_31_12}
\begin{aligned}
-\sum_{t,i,k}\left<n_{ik,t}+L_{i}[k,:]\lambda^{*}_{i,t},x^{*}_{i,t}-\bar x_{i,t}\right>\leq 0.
\end{aligned}
\end{equation}

Substituing (\ref{lemm_31_14}), (\ref{lemm_31_11}) and (\ref{lemm_31_12}) into (\ref{lemm_31_10}), one obtains
\begin{equation}\label{lemm_31_66}
\begin{aligned}
&\sum_{t,i,k}\bigg<\sum_{s \in \mathcal S_{ik,t}}y^{s}_{ik,t},x^{*}_{i,t}-\bar x_{i,t}\bigg>
\\& \leq b_{4}+n\vartheta b_{1}G+(b_{5}+n\vartheta b_{2}G)T^{\tau-a_{1}}+b_{6}T^{1+\tau+a_{2}-a_{1}-a_{3}}
\\&\quad+nG(2R+1)T^{u+\phi}+(2cKR+\vartheta b_{3})nGT^{2+a_{2}-a_{3}}
\\&\quad-\sum_{t,i,k}\left<\nabla f^{i,t}_{k}\left( x^{*}_{t}\right)-\nabla f^{i,t}_{k}\left(\bar x_{t}\right),x^{*}_{i,t}-\bar x_{i,t}\right>+\vartheta\mathcal CV(T).
\end{aligned}
\end{equation}

Secondly, consider the upper bound of the term $\sum_{t,i,k}\left<\sum_{s \in \mathcal S_{ik,t}}y^{s}_{ik,t},\bar x_{i,t}-\bar{x}_{i,t+1}\right>$ on the right-hand side of  (\ref{lemm_31_50}). By utilizing Young's inequality, one has
\begin{equation}\label{lemm_31_6}
\begin{aligned}
&\sum_{t,i,k}\left<\sum_{s \in \mathcal S_{ik,t}}y^{s}_{ik,t},\bar x_{i,t}-\bar x_{i,t+1}\right>
\\& \leq \frac{cG^{2}}{2}\sum_{t,i,k}\alpha_{t}(1+\left\|\mu_{ik,t}\right\|)^{2}+\sum_{t=1}^{T}\sum_{i=1}^{N}\frac{n_{i}}{2\alpha_{t}}\left\|\bar x_{i,t}-\bar x_{i,t+1}\right\|^{2}
\\&\leq \frac{ncG^{2}}{2}T^{1-a_{1}}(1+2KT^{1+a_{2}-a_{3}}+K^{2}T^{2+2a_{2}-2a_{3}})
\\&\quad+\sum_{t=1}^{T}\sum_{i=1}^{N}\frac{n_{i}}{2\alpha_{t}}\left\|\bar x_{i,t}-\bar x_{i,t+1}\right\|^{2}.
\end{aligned}
\end{equation}

Thirdly, consider the upper bound of the term $\sum_{t=1}^{T}\frac{1}{\alpha_{t}}\sum_{i=1}^{N}\sum_{k=1}^{n_{i}}\left<{\theta}_{ik,t+1},\bar x_{i,t+1}-x_{ik,t+1}\right>$ on the right-hand side of (\ref{lemm_31_50}). From  (\ref{eq.lemm_46}), it follows that
\begin{equation}\label{lemm_31_67}
\begin{aligned}
&\sum_{t=1}^{T}\frac{1}{\alpha_{t}}\sum_{i=1}^{N}\sum_{k=1}^{n_{i}}\left<{\theta}_{ik,t+1},\bar x_{i,t+1}-x_{ik,t+1}\right>
\\&\leq\sum_{t=1}^{T}\mathop{\max}\limits_{i \in \mathcal V, k \in [n_{i}]}\left\|\bar x_{i,t+1}-x_{ik,t+1}\right\|\left(\frac{1}{\alpha_{t}}\sum_{i=1}^{N}\sum_{k=1}^{n_{i}}\left\|{\theta}_{ik,t+1}\right\|\right)
\\&\leq 2cnG\big((2R+b_{1})+b_{2}T^{1-a_{1}}+b_{3}KT^{3+2a_{2}-2a_{3}-a_{1}}
\\&\quad+(K(b_{1}+b_{2}+2R)+b_{3})T^{2+a_{2}-a_{1}-a_{3}}\big).
\end{aligned}
\end{equation}

Substituting inequalities (\ref{lemm_31_66}), (\ref{lemm_31_6}) and (\ref{lemm_31_67}) into (\ref{lemm_31_50}) and  regrouping the terms, one obtains
\begin{equation}\label{lemm_31_68}
\begin{aligned}
&\sum_{t,i,k}\left<\nabla f^{i,t}_{k}\left( x^{*}_{t}\right)-\nabla f^{i,t}_{k}\left(\bar x_{t}\right),x^{*}_{i,t}-\bar x_{i,t}\right>
\\& \leq \sum_{t=1}^{T}\sum_{i=1}^{N}\frac{n_{i}}{2\alpha_{t}}(\left\|\bar{x}_{i,t}-x^{*}_{i,t}\right\|^{2}-\left\|\bar{x}_{i,t+1}-x^{*}_{i,t+1}\right\|^{2})
\\& \quad +\frac{2nR}{\alpha_{T}}\Phi_{T}+\vartheta\mathcal CV(T)+b_{7}+b_{8}T^{\tau-a_{1}}+b_{9}T^{1+\tau+a_{2}-a_{1}-a_{3}}
\\&\quad +b_{10}T^{u+\phi}+b_{11}T^{2+a_{2}-a_{3}}+b_{12}T^{3+2a_{2}-2a_{3}-a_{1}},
\end{aligned}
\end{equation}
where $b_{7}=b_{4}+2cnG(2R+b_{1})+n\vartheta b_{1}G$, $b_{8}=b_{5}+2cnGb_{2}+n\vartheta b_{2}G+0.5ncG^{2}$, $b_{9}=b_{6}+2cnG(K(b_{1}+b_{2}+2R)+b_{3})+ncG^{2}K$, $b_{10}=nG(2R+1)$, $b_{11}=(2cKR+\vartheta b_{3})nG$ and $b_{12}=2cnb_{3}GK+0.5ncG^{2}K^{2}$.

Note that
\begin{align*}
&\sum_{t=1}^{T}\sum_{i=1}^{N}\frac{n_{i}}{2\alpha_{t}}(\left\|\bar{x}_{i,t}-x^{*}_{i,t}\right\|^{2}-\left\|\bar{x}_{i,t+1}-x^{*}_{i,t+1}\right\|^{2})
\\&\leq \frac{T^{a_{1}}}{2}\sum_{i=1}^{N}n_{i}\left\|\bar{x}_{i,1}-x^{*}_{i,1}\right\|^{2}
\\&\leq 2nR^{2}T^{a_{1}}.
\end{align*} 
Then, by $\sum_{t=1}^{T}\left\|\bar x_{i,t}-x^{*}_{i,t}\right\|\leq \sqrt{T\sum_{t=1}^{T}\left\|\bar x_{t}-x^{*}_{t}\right\|^{2}}$ , Assumption \ref{assp.two} and dividing $\sigma$ on both sides of (\ref{lemm_31_68}), it follows that
\begin{align*}
&\sum_{t=1}^{T}\left\|\bar x_{i,t}-x^{*}_{i,t}\right\|
\\&\leq \sqrt{\frac{b_{10}}{\sigma}}T^{\frac{1+u+\phi}{2}}+\left(\sqrt{\frac{b_{7}}{\sigma}}+\sqrt{\frac{2nR^{2}}{\sigma}}+\sqrt{\frac{2nR}{\sigma}}\right)T^{\frac{1+a_{1}+\phi}{2}}
\\&\quad+\sqrt{\frac{b_{8}}{\sigma}}T^{\frac{1+\tau-a_{1}}{2}}+\sqrt{\frac{b_{11}}{\sigma}}T^{\frac{3+a_{2}-a_{3}}{2}}+\sqrt{\frac{b_{12}}{\sigma}}T^{2+\frac{2a_{2}-a_{1}-2a_{3}}{2}}
\\&\quad+\sqrt{\frac{b_{9}}{\sigma}}T^{\frac{2+\tau+a_{2}-a_{1}-a_{3}}{2}}+\sqrt{\frac{\vartheta\mathcal CV(T)T}{\sigma}}.
\end{align*}

The proof is complete.

\subsection{Proof of Theorem 1}
It is evident that the stepsizes $\alpha_{t}$, $\beta_{t}$ and $\gamma_{t}$ setting in Theorem \ref{thm1} satisfy Assumption \ref{assp.seven}. Hence, for $t \geq 1$, by substituting $\beta_{t}$ and $\gamma_{t}$ into Lemma \ref{lemm_3}, one obtains 
\begin{align*}
&-\sum_{t=1}^{T}\sum_{i=1}^{N}\sum_{j=1}^{n_{i}}\left(\left[g^{i,t}_{j}\left(x_{ij,t}\right)\right]_{+}\right)^\top\mu-\frac{n}{2}\left\|\mu\right\|^{2}
\\&\leq-\sum_{t=1}^{T}\sum_{i=1}^{N}\sum_{j=1}^{n_{i}}\sum_{s \in \mathcal S_{ij,t}}\left(\left[g^{i,s}_{j}\left(x_{ij,t}\right)\right]_{+}\right)^\top\mu-\frac{n}{2}\left\|\mu\right\|^{2}
\\&\leq b_{13}{T^{1-2a_{3}}}+b_{14}T^{2+2a_{2}-4a_{3}}+b_{15}T^{2+2a_{2}-3a_{3}},
\end{align*}
where $b_{13}=ns^{2}K^{2}$, $b_{14}=\frac{64}{15}nK^{2}$ and $b_{15}=\frac{40n^{\frac{3}{2}}K^{2}}{3\left(1-\lambda\right)}$.

Submitting $\mu = -\frac{\sum_{t=1}^{T}\sum_{i=1}^{N}\sum_{j=1}^{n_{i}}\left[g^{i,t}_{j}\left(x_{ij,t}\right)\right]_{+}}{n}$ into the above inequality yields
\begin{align*}
\mathcal CV(T)&=\left\|\sum_{t=1}^{T}\sum_{i=1}^{N}\sum_{j=1}^{n_{i}}\left[g^{i,t}_{j}(x_{ij,t})\right]_{+}\right\|
\\&\leq \sqrt{2n}\left(\sqrt{b_{13}}{T^{\frac{1}{2}-a_{3}}}+(\sqrt{b_{14}}+\sqrt{b_{15}})T^{1+a_{2}-\frac{3}{2}a_{3}}\right),
\end{align*}
Thus, one arrives at the second result.
By utilizing the convexity of cost functions and combining Lemma \ref{lemma_6} and Lemma \ref{lemma_7}, one obtains 
\begin{align*}
\mathcal R(T)&=\sum_{t=1}^{T}\sum_{i=1}^{N}\sum_{j=1}^{n_{i}}\left(f^{i,t}_{j}\left(x_{ij,t},x_{-i,t}^{*}\right)-f^{i,t}_{j}\left(x_{i,t}^{*},x_{-i,t}^{*}\right)\right)
\\&\leq G\sum_{i=1}^{N}\sum_{j=1}^{n_{i}}(\sum_{t=1}^{T}\left\|x_{ij,t}-\bar x_{i,t}\right\|+\sum_{t=1}^{T}\left\|\bar x_{i,t}-x_{i,t}^{*}\right\|)
\\&\leq nG(b_{1}+b_{2}T^{1-a_{1}}+b_{3}T^{2+a_{2}-a_{1}-a_{3}}+b_{16}T^{\frac{1+u+\phi}{2}}
\\&\quad+b_{17}T^{\frac{1+a_{1}+\phi}{2}}+b_{18}T^{\frac{1+\tau-a_{1}}{2}}+b_{19}T^{\frac{3+a_{2}-a_{3}}{2}}
\\&\quad+b_{20}T^{2+\frac{2a_{2}-a_{1}-2a_{3}}{2}}+b_{21}T^{\frac{2+\tau+a_{2}-a_{1}-a_{3}}{2}}),
\end{align*}
where $b_{16}=\sqrt{\frac{b_{10}}{\sigma}}$, $b_{17}=\sqrt{\frac{b_{7}}{\sigma}}+\sqrt{\frac{2nR^{2}}{\sigma}}+\sqrt{\frac{2nR}{\sigma}}$, $b_{18}=\sqrt{\frac{b_{8}}{\sigma}}+\sqrt{\frac{\vartheta\sqrt{2nb_{13}}}{\sigma}}$, $b_{19}=\sqrt{\frac{b_{11}}{\sigma}}+\sqrt{\frac{\vartheta\sqrt{2n}(\sqrt{b_{14}}+\sqrt{b_{15}})}{\sigma}}$, $b_{20}=\sqrt{\frac{b_{12}}{\sigma}}$, $b_{21}=\sqrt{\frac{b_{9}}{\sigma}}$. The first result is obtained.
 The proof is complete.
\end{appendix}

\end{document}